\documentclass[11pt]{amsart}
\usepackage[colorinlistoftodos]{todonotes}
\usepackage{amsthm,amssymb,amsmath}
\usepackage{graphicx}
\usepackage[english]{babel}
\usepackage{tikz-network}
\newtheorem{theorem}{Theorem}[section]

\newtheorem{conjecture}{Conjecture}[section]
\newtheorem{corollary}{Corollary}[theorem]
\newtheorem{lemma}[theorem]{Lemma}
\newtheorem{definition}{Definition}[section]
\theoremstyle{definition}
\theoremstyle{remark}
\newtheorem*{remark}{Remark}

\newcommand{\eps}{\varepsilon}
\usepackage[letterpaper,top=2cm,bottom=3cm,left=4cm,right=4cm,marginparwidth=2cm]{geometry}

\usepackage[colorlinks=true, allcolors=blue]{hyperref}
\usepackage{mathtools}

\title[Strict Inequalities]{Strict Inequalities for the $n$-crossing Number}
\author[N. Hagedorn]{Nicholas Hagedorn}

\begin{document}

\begin{abstract}

In 2013, Adams introduced the $n$-crossing number of a knot $K$, denoted by $c_n(K)$. Inequalities between the $2$-, $3$-, $4$-, and $5$-crossing numbers have been previously established. We prove $c_9(K)\leq c_3(K)-2$ for all knots $K$ that are not the trivial, trefoil, or figure-eight knot. We show this inequality is optimal and obtain previously unknown values of $c_9(K)$. We generalize this inequality to prove that $c_{13}(K) < c_{5}(K)$ for a certain set of classes of knots.
\end{abstract}

\maketitle

\section{Introduction}

In knot theory, a standard crossing in a knot diagram is when one strand passes over another. In \cite{adams2013}, Adams introduced an $n$-crossing projection, which is a knot diagram where at each crossing $n$ strands intersect. Every knot has an $n$-crossing projection for all $n\geq 2$. One can then define $c_n(K)$ as the smallest number of crossings an $n$-crossing projection of $K$ can contain.\\

Inequalities between these $n$-crossing numbers have been studied: The construction in Figure~\ref{increasebytwofirsttime} shows that every crossing of $n$ strands can be turned into one with $n+2$ strands. Thus, an $n$-crossing can be made into an $(n+2)$-crossing and $c_{n+2}(K)\leq c_{n}(K)$. Interestingly, $c_3(K)\leq c_2(K)-1$ for 2-braid knots $K$ and $c_3(K)\leq c_2(K)-2$ for all other non-trivial knots \cite{adams2013}. For a non-trivial knot $K$, $c_4(K)\leq c_2(K) - 1$ \cite{fourcrossing} and $c_5(K)\leq c_3(K)-1$ \cite{tripmoves}.\\

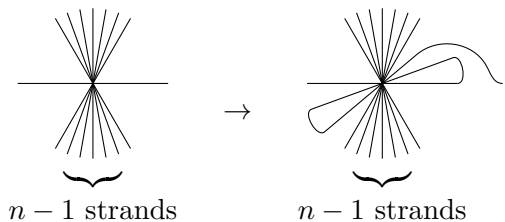
\begin{figure}[!htb]
        \center{
        \center{\begin{tabular}{lll}
    \raisebox{-.5\height}{
\begin{tikzpicture}
\coordinate (Center) at (0, 0);

\coordinate (1) at (0:1);
\coordinate (2) at (60:1);
\coordinate (3) at (120:1);
\coordinate (4) at (180:1);
\coordinate (5) at (240:1);
\coordinate (6) at (300:1);
\node at (Center) [below = 10mm of Center] {$\underbrace{\quad\quad}$};
\node at (Center) [below = 14mm of Center] {$n-1$ strands};

\coordinate(n-0) at (70:1);
\coordinate(n-1) at (80:1);
\coordinate(n-2) at (90:1);
\coordinate(n-3) at (100:1);
\coordinate(n-4) at (110:1);

\coordinate(n-5) at (250:1);
\coordinate(n-6) at (260:1);
\coordinate(n-7) at (270:1);
\coordinate(n-8) at (280:1);
\coordinate(n-9) at (290:1);

\draw (n-0) -- (n-5);
\draw (n-1) -- (n-6);
\draw (n-2) -- (n-7);
\draw (n-3) -- (n-8);
\draw (n-4) -- (n-9);

\draw (Center) -- (1);
\draw (Center) -- (2);
\draw (Center) -- (3);
\draw (Center) -- (4);
\draw (Center) -- (5);
\draw (Center) -- (6);

\end{tikzpicture}

} & $\to$  & \raisebox{-.5\height}{

\begin{tikzpicture}
\coordinate (Center) at (0, 0);

\coordinate (1) at (0:1);
\coordinate (2) at (60:1);
\coordinate (3) at (120:1);
\coordinate (4) at (180:1);
\coordinate (5) at (240:1);
\coordinate (6) at (300:1);
\coordinate (loop1) at (20:1);
\coordinate (loop2) at (200:1);
\coordinate (loop3) at (220:1);
\coordinate (loop4) at (40:0.6);
\coordinate (loop5) at (10:1.4);
\coordinate (loop6) at (0:1.6);

\node at (Center) [below = 10mm of Center] {$\underbrace{\quad\quad}$};
\node at (Center) [below = 14mm of Center] {$n-1$ strands};

\coordinate(n-0) at (70:1);
\coordinate(n-1) at (80:1);
\coordinate(n-2) at (90:1);
\coordinate(n-3) at (100:1);
\coordinate(n-4) at (110:1);

\coordinate(n-5) at (250:1);
\coordinate(n-6) at (260:1);
\coordinate(n-7) at (270:1);
\coordinate(n-8) at (280:1);
\coordinate(n-9) at (290:1);

\draw (n-0) -- (n-5);
\draw (n-1) -- (n-6);
\draw (n-2) -- (n-7);
\draw (n-3) -- (n-8);
\draw (n-4) -- (n-9);

\draw (Center) -- (1);
\draw (Center) -- (2);
\draw (Center) -- (3);
\draw (Center) -- (4);
\draw (Center) -- (5);
\draw (Center) -- (6);

\draw (1) to [out=0,in=20] (loop1) to (loop2) to [out=20+180, in=40+180] (loop3) to (loop4) to [out=40, in=180-60](loop5) to [out=-60, in=180](loop6);

\end{tikzpicture}}\\
\end{tabular}}
        }
        \caption{\label{increasebytwofirsttime} Turning an $n$-crossing into an $(n+2)$-crossing}
\end{figure}

Until now, specific $n$-crossing numbers for $n>5$ have not been studied. In Section~\ref{sectionclassification}, we deduce upper bounds on $c_9(K)$ for knots with certain 5-crossing projections and obtain new values of $c_9(K)$. We prove that $c_9(K)\leq c_3(K)-2$ for all knots $K$ that are not the trivial, trefoil, or figure-eight knot in Section~\ref{section9crossing}. We also show this inequality is optimal. In Section~\ref{section13crossing}, we classify all knots based on their crossing connections. We then show $c_{13}(K)<c_5(K)$ for a majority of these classifications.\\

\section{Classification of 5-crossing Knots}\label{sectionclassification}

Adams introduced the notion of a crossing covering circle in \cite{adams2013}. For some knot projection $P$, a crossing covering circle is a topological circle in the projection plane that only intersects $P$ at crossings, and, when it intersects a crossing, it passes straight through the crossing. We introduce two definitions that generalize this concept.

\begin{definition}[Crossing Segment]
Take an $n$-crossing knot projection $P$. Define a crossing segment to be a topological line segment in the projection plane that only intersects $P$ at crossings. When the crossing segment does intersect $P$, it must pass directly through the crossing, with $n$ strands of the crossing on either side of $P$. Lastly, a crossing segment cannot intersect a crossing more than once.
\end{definition}

\begin{definition}[Crossing Connected]
For a given $n$-crossing knot projection $P$, we say two faces in $P$ are crossing connected via $\alpha$ if they have common vertex $\alpha$ (a crossing) and can be connected by a crossing segment that does not intersect $\alpha$. Faces are adjacent crossing connected (ACC) via $\alpha$ if they are crossing connected via $\alpha$ and both border a strand that connects $\alpha$ to another crossing. Faces are opposite crossing connected (OCC) via $\alpha$ if they are crossing connected via $\alpha$ and are on opposite sides of $\alpha$.
\end{definition}

Note that the crossing segment connecting two opposite crossing connected faces is equivalent to a crossing covering circle by extending the crossing segment through the crossing. Next, we will need the following lemma to gain insights into ACC and OCC faces.

\begin{lemma}\label{2colorable}
If all vertices of a planar graph have even degrees, then the graph's dual graph is $2$-colorable.
\end{lemma}
\begin{proof}
Let $D$, the dual graph of the planar graph $G$, be embedded in the same plane as $G$. Since a graph is 2-colorable precisely when there does not exist an odd cycle, it is enough to show that $D$ does not contain an odd cycle. We will assume such a cycle exists, and take $(d_1, d_2, \dots, d_n, d_1)$ to be the cycle in $D$ with minimum odd $n$. Consider the union of cycle $(d_1, d_2, \dots, d_n, d_1)$ with $G$. Let vertex $h_i$ denote the intersection of the edge $(d_i, d_{i+1})$ with $G$ for $1\leq i<n$, and let $h_n$ denote the intersection of edge $(d_n, d_1)$ with $G$. \\

We now define a new undirected graph $H$. Let the vertices of $H$ consist of the vertices $d_i$, $h_i$, $1\leq i\leq n$, and all the vertices of $G$ that lie inside the cycle $d=(d_1, d_2, \dots, d_n, d_1)$. Each vertex $h_i$ divides an edge of $d$ into two edges $e_{i,1}$, $e_{i,2}$.  It also divides an edge of $G$ into two parts, one of which lies inside of $d$. This edge will be called $e_{i,3}$. Let the edges of $H$ consist of the edges $e_{i,j}$ and the edges of $G$ lying inside $d$. \\

The handshaking lemma gives that $\sum_{v\in H}\deg(v)$ is even. But each vertex $h_i$ has degree 3. With an odd number of vertices $h_i$, at least one of the vertices $v\in G$ must have an odd degree. This violates the given, the desired contradiction.
\end{proof}

\begin{lemma}\label{oddcrossinters}
Let $n$ be odd. In an $n$-crossing projection of a knot, a crossing segment connecting two ACC or OCC faces passes through an odd number of crossings.
\end{lemma}
\begin{proof}
Every crossing in an $n$-crossing knot projection has degree $2n$. So, Lemma \ref{2colorable} gives that there exists a checkerboard coloring for the faces of any $n$-crossing knot projection. It is easy to see that opposite and adjacent sides are different colors for odd $n$. And since, for odd $n$, face color alternates when the crossing segment passes through a crossing, a crossing segment connecting two ACC or OCC faces passes through an odd number of crossings.
\end{proof}

\begin{theorem}\label{accocctheorem}
Given an $n$-crossing projection of knot $K$ with odd $n$ and two adjacent or opposite crossing connected faces, there exists a $(2n-1)$-crossing diagram of $K$ with one fewer crossing.
\end{theorem}

\begin{proof}
Regardless of whether there is an ACC or OCC, we know the crossing segment intersects an odd number of crossings due to Lemma~\ref{oddcrossinters}. Then, if two faces are ACC via a crossing $\alpha$, we can pull the $n-1$ strands of $\alpha$ that don't border both ACC faces around the crossing segment. This move is shown in Figure~\ref{accineq} for the $n=5$ case. Similarly, if two faces are OCC via a crossing $\alpha$, we can pull the $n-2$ strands of $\alpha$ that don't border both OCC faces and one of $\alpha$'s two strands that do border $\alpha$ around the crossing segment. This move is shown in Figure~\ref{occineq} for the $n=5$ case. Both moves eliminate a crossing while turning the $n$-crossings along the crossing segment into $(2n-1)$-crossings. We can then perform the move shown in Figure~\ref{increasebytwofirsttime} a total of $\frac{n-1}{2}$ times on all remaining crossings to convert them from $n$-crossings to $(2n-1)$-crossings. This construction provides a $(2n-1)$-projection of $K$ with one fewer crossing than the original $n$-crossing projection.
\end{proof}

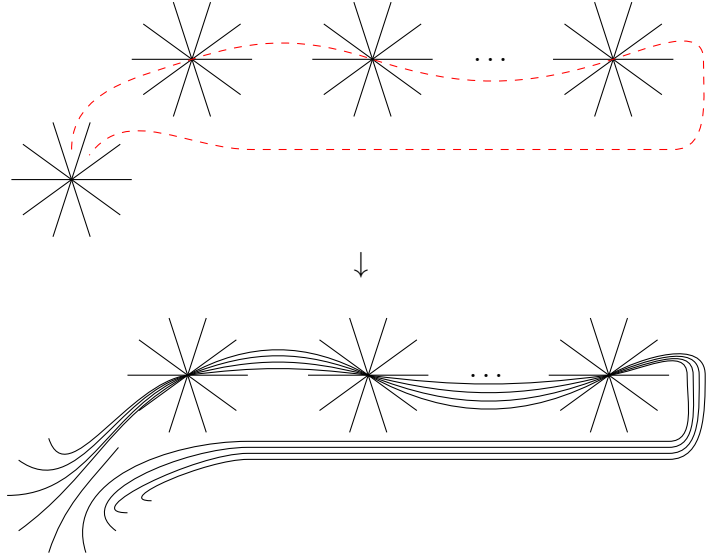
\begin{figure}[!htb]
        \center{
\begin{tikzpicture}[scale=0.8]

\foreach \position in {-3, 0, 4} {
    \coordinate (Center-\position) at (0+\position, 0);

    \coordinate (1-\position) at (0.809+\position, 0.588);
    \coordinate (2-\position) at (0.309+\position, 0.951);
    \coordinate (3-\position) at (-0.309+\position, 0.951);
    \coordinate (4-\position) at (-0.809+\position,0.588);
    \coordinate (5-\position) at (-1+\position,0);
    \coordinate (6-\position) at (-0.809+\position,-.588);
    \coordinate (7-\position) at (-0.309+\position,-.951);
    \coordinate (8-\position) at (0.309+\position,-.951);
    \coordinate (9-\position) at (0.809+\position,-.588);
    \coordinate (10-\position) at (1+\position,0);

    \draw (Center-\position) -- (1-\position);
    \draw (Center-\position) -- (2-\position);
    \draw (Center-\position) -- (3-\position);
    \draw (Center-\position) -- (4-\position);
    \draw (Center-\position) -- (5-\position);
    \draw (Center-\position) -- (6-\position);
    \draw (Center-\position) -- (7-\position);
    \draw (Center-\position) -- (8-\position);
    \draw (Center-\position) -- (9-\position);
    \draw (Center-\position) -- (10-\position);
}
\foreach \position in {-5}{
    \coordinate (Center-\position) at (0+\position, 0-2);

    \coordinate (1-\position) at (0.809+\position, 0.588-2);
    \coordinate (2-\position) at (0.309+\position, 0.951-2);
    \coordinate (3-\position) at (-0.309+\position, 0.951-2);
    \coordinate (4-\position) at (-0.809+\position,0.588-2);
    \coordinate (5-\position) at (-1+\position,0-2);
    \coordinate (6-\position) at (-0.809+\position,-.588-2);
    \coordinate (7-\position) at (-0.309+\position,-.951-2);
    \coordinate (8-\position) at (0.309+\position,-.951-2);
    \coordinate (9-\position) at (0.809+\position,-.588-2);
    \coordinate (10-\position) at (1+\position,0-2);
    
    \draw (Center-\position) -- (1-\position);
    \draw (Center-\position) -- (2-\position);
    \draw (Center-\position) -- (3-\position);
    \draw (Center-\position) -- (4-\position);
    \draw (Center-\position) -- (5-\position);
    \draw (Center-\position) -- (6-\position);
    \draw (Center-\position) -- (7-\position);
    \draw (Center-\position) -- (8-\position);
    \draw (Center-\position) -- (9-\position);
    \draw (Center-\position) -- (10-\position);
}

\node at ($(Center-0)!.5!(Center-4)$) {\ldots};

\draw [dashed, red] (-5, -2 + 0.5) to [out=90, in=180+18] (Center--3) to [out=18, in=180-18] (Center-0) to [out=-18, in=180+18] (Center-4) to [out=18, in=90] (5.5, 0) to [out=270, in=0] (5, -1.5) to (-2, -1.5) to [out=180, in=54] (0.2938-5, 0.4045-2);

\end{tikzpicture}

$\downarrow$\\
\vspace{5mm}

\begin{tikzpicture}[scale=0.8]

\foreach \position in {-3, 0, 4} {
    \coordinate (Center-\position) at (0+\position, 0);

    \coordinate (1-\position) at (0.809+\position, 0.588);
    \coordinate (2-\position) at (0.309+\position, 0.951);
    \coordinate (3-\position) at (-0.309+\position, 0.951);
    \coordinate (4-\position) at (-0.809+\position,0.588);
    \coordinate (5-\position) at (-1+\position,0);
    \coordinate (6-\position) at (-0.809+\position,-.588);
    \coordinate (7-\position) at (-0.309+\position,-.951);
    \coordinate (8-\position) at (0.309+\position,-.951);
    \coordinate (9-\position) at (0.809+\position,-.588);
    \coordinate (10-\position) at (1+\position,0);

    \draw (Center-\position) -- (1-\position);
    \draw (Center-\position) -- (2-\position);
    \draw (Center-\position) -- (3-\position);
    \draw (Center-\position) -- (4-\position);
    \draw (Center-\position) -- (5-\position);
    \draw (Center-\position) -- (6-\position);
    \draw (Center-\position) -- (7-\position);
    \draw (Center-\position) -- (8-\position);
    \draw (Center-\position) -- (9-\position);
    \draw (Center-\position) -- (10-\position);
}
\foreach \position in {-5}{
    \coordinate (Center-\position) at (0+\position, 0-2);

    \coordinate (1-\position) at (0.809+\position, 0.588-2);
    \coordinate (2-\position) at (0.309+\position, 0.951-2);
    \coordinate (3-\position) at (-0.309+\position, 0.951-2);
    \coordinate (4-\position) at (-0.809+\position,0.588-2);
    \coordinate (5-\position) at (-1+\position,0-2);
    \coordinate (6-\position) at (-0.809+\position,-.588-2);
    \coordinate (7-\position) at (-0.309+\position,-.951-2);
    \coordinate (8-\position) at (0.309+\position,-.951-2);
    \coordinate (9-\position) at (0.809+\position,-.588-2);
    \coordinate (10-\position) at (1+\position,0-2);
    
    \draw (7-\position) to [out=72, in=180+50] (0.85+\position, 0.798-2);
}

\node at ($(Center-0)!.5!(Center-4)$) {\ldots};

\draw (-0.309*1-5, 0.951*1-2) to [out=-72, in=180+7.2] (Center--3) to [out=7.2, in=180-7.2] (Center-0) to [out=-7.2, in=180+7.2] (Center-4) to [out=7.2, in=90] (5.3, 0) to [out=270, in=0] (5, -1.1) to (-2, -1.1) to [out=180, in=180-72] (8--5);

\draw (-0.809*1-5, 0.588*1-2) to [out=-36, in=180+14.4] (Center--3) to [out=14.4, in=180-14.4] (Center-0) to [out=-14.4, in=180+14.4] (Center-4) to [out=14.4, in=90] (5.4, 0) to [out=270, in=0] (5, -1.2) to (-2, -1.2) to [out=180, in=180-36] (9--5);

\draw (-1*1-5, 0*1-2) to [out=0, in=180+21.6] (Center--3) to [out=21.6, in=180-21.6] (Center-0) to [out=-21.6, in=180+21.6] (Center-4) to [out=21.6, in=90] (5.5, 0) to [out=270, in=0] (5, -1.3) to (-2, -1.3) to [out=180, in=180] (1-5, -0.288-2);

\draw (-0.809*1-5, -0.588*1-2) to [out=36, in=180+28.8] (Center--3) to [out=28.8, in=180-28.8] (Center-0) to [out=-28.8, in=180+28.8] (Center-4) to [out=28.8, in=90] (5.6, 0) to [out=270, in=0] (5, -1.4) to (-2, -1.4) to [out=180, in=180] (1.4-5, -0.088-2);

\end{tikzpicture}
        }
        \caption{\label{accineq} An $n$-crossing projection with two ACC faces can be transformed into an $(n+2)$-crossing projection with one fewer crossing}
\end{figure}

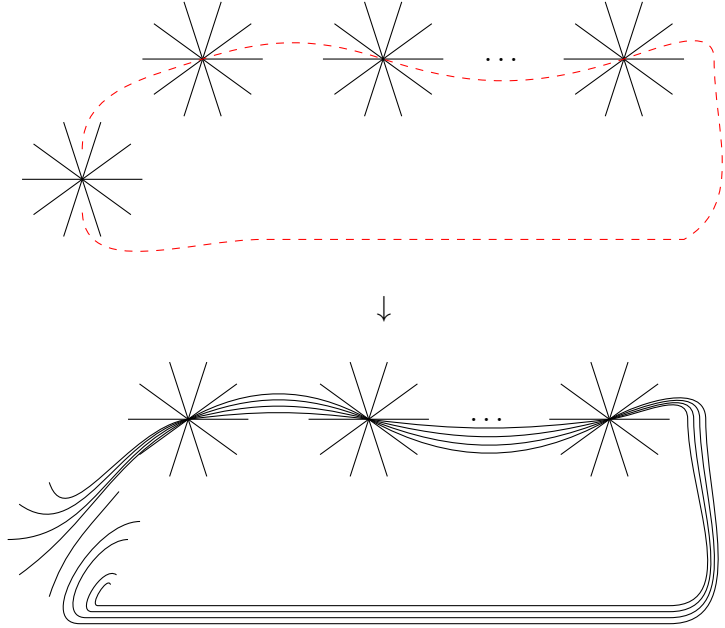
\begin{figure}[!htb]
        \center{

\begin{tikzpicture}[scale=0.8]

\foreach \position in {-3, 0, 4} {
    \coordinate (Center-\position) at (0+\position, 0);

    \coordinate (1-\position) at (0.809+\position, 0.588);
    \coordinate (2-\position) at (0.309+\position, 0.951);
    \coordinate (3-\position) at (-0.309+\position, 0.951);
    \coordinate (4-\position) at (-0.809+\position,0.588);
    \coordinate (5-\position) at (-1+\position,0);
    \coordinate (6-\position) at (-0.809+\position,-.588);
    \coordinate (7-\position) at (-0.309+\position,-.951);
    \coordinate (8-\position) at (0.309+\position,-.951);
    \coordinate (9-\position) at (0.809+\position,-.588);
    \coordinate (10-\position) at (1+\position,0);

    \draw (Center-\position) -- (1-\position);
    \draw (Center-\position) -- (2-\position);
    \draw (Center-\position) -- (3-\position);
    \draw (Center-\position) -- (4-\position);
    \draw (Center-\position) -- (5-\position);
    \draw (Center-\position) -- (6-\position);
    \draw (Center-\position) -- (7-\position);
    \draw (Center-\position) -- (8-\position);
    \draw (Center-\position) -- (9-\position);
    \draw (Center-\position) -- (10-\position);
}
\foreach \position in {-5}{
    \coordinate (Center-\position) at (0+\position, 0-2);

    \coordinate (1-\position) at (0.809+\position, 0.588-2);
    \coordinate (2-\position) at (0.309+\position, 0.951-2);
    \coordinate (3-\position) at (-0.309+\position, 0.951-2);
    \coordinate (4-\position) at (-0.809+\position,0.588-2);
    \coordinate (5-\position) at (-1+\position,0-2);
    \coordinate (6-\position) at (-0.809+\position,-.588-2);
    \coordinate (7-\position) at (-0.309+\position,-.951-2);
    \coordinate (8-\position) at (0.309+\position,-.951-2);
    \coordinate (9-\position) at (0.809+\position,-.588-2);
    \coordinate (10-\position) at (1+\position,0-2);
    
    \draw (Center-\position) -- (1-\position);
    \draw (Center-\position) -- (2-\position);
    \draw (Center-\position) -- (3-\position);
    \draw (Center-\position) -- (4-\position);
    \draw (Center-\position) -- (5-\position);
    \draw (Center-\position) -- (6-\position);
    \draw (Center-\position) -- (7-\position);
    \draw (Center-\position) -- (8-\position);
    \draw (Center-\position) -- (9-\position);
    \draw (Center-\position) -- (10-\position);
}

\node at ($(Center-0)!.5!(Center-4)$) {\ldots};

\draw [dashed, red] (-5, -2 + 0.5) to [out=90, in=180+18] (Center--3) to [out=18, in=180-18] (Center-0) to [out=-18, in=180+18] (Center-4) to [out=18, in=90] (5.5, 0) to [out=270, in=30] (5, -3) to (-2, -3) to [out=180, in=270] (-5, -2 - 0.5);

\end{tikzpicture}

$\downarrow$\\
\vspace{5mm}

\begin{tikzpicture}[scale=0.8]

\foreach \position in {-3, 0, 4} {
    \coordinate (Center-\position) at (0+\position, 0);

    \coordinate (1-\position) at (0.809+\position, 0.588);
    \coordinate (2-\position) at (0.309+\position, 0.951);
    \coordinate (3-\position) at (-0.309+\position, 0.951);
    \coordinate (4-\position) at (-0.809+\position,0.588);
    \coordinate (5-\position) at (-1+\position,0);
    \coordinate (6-\position) at (-0.809+\position,-.588);
    \coordinate (7-\position) at (-0.309+\position,-.951);
    \coordinate (8-\position) at (0.309+\position,-.951);
    \coordinate (9-\position) at (0.809+\position,-.588);
    \coordinate (10-\position) at (1+\position,0);

    \draw (Center-\position) -- (1-\position);
    \draw (Center-\position) -- (2-\position);
    \draw (Center-\position) -- (3-\position);
    \draw (Center-\position) -- (4-\position);
    \draw (Center-\position) -- (5-\position);
    \draw (Center-\position) -- (6-\position);
    \draw (Center-\position) -- (7-\position);
    \draw (Center-\position) -- (8-\position);
    \draw (Center-\position) -- (9-\position);
    \draw (Center-\position) -- (10-\position);
}
\foreach \position in {-5}{
    \coordinate (Center-\position) at (0+\position, 0-2);

    \coordinate (1-\position) at (0.809+\position, 0.588-2);
    \coordinate (2-\position) at (0.309+\position, 0.951-2);
    \coordinate (3-\position) at (-0.309+\position, 0.951-2);
    \coordinate (4-\position) at (-0.809+\position,0.588-2);
    \coordinate (5-\position) at (-1+\position,0-2);
    \coordinate (6-\position) at (-0.809+\position,-.588-2);
    \coordinate (7-\position) at (-0.309+\position,-.951-2);
    \coordinate (8-\position) at (0.309+\position,-.951-2);
    \coordinate (9-\position) at (0.809+\position,-.588-2);
    \coordinate (10-\position) at (1+\position,0-2);
    
    \draw (7-\position) to [out=72, in=180+50] (0.85+\position, 0.798-2);
}

\node at ($(Center-0)!.5!(Center-4)$) {\ldots};

\draw (-0.309*1-5, 0.951*1-2) to [out=-72, in=180+7.2] (Center--3) to [out=7.2, in=180-7.2] (Center-0) to [out=-7.2, in=180+7.2] (Center-4) to [out=7.2, in=90] (5.3, 0) to [out=270, in=0] (5, -3.1) to (-4.5, -3.1) to [out=180, in=180-72] (0.709-5,-.751-2);

\draw (-0.809*1-5, 0.588*1-2) to [out=-36, in=180+14.4] (Center--3) to [out=14.4, in=180-14.4] (Center-0) to [out=-14.4, in=180+14.4] (Center-4) to [out=14.4, in=90] (5.4, 0) to [out=270, in=0] (5, -3.2) to (-4.6, -3.2) to [out=180, in=180-36] (9--5);

\draw (-1*1-5, 0*1-2) to [out=0, in=180+21.6] (Center--3) to [out=21.6, in=180-21.6] (Center-0) to [out=-21.6, in=180+21.6] (Center-4) to [out=21.6, in=90] (5.5, 0) to [out=270, in=0] (5, -3.3) to (-4.7, -3.3) to [out=180, in=180] (10--5);

\draw (-0.809*1-5, -0.588*1-2) to [out=36, in=180+28.8] (Center--3) to [out=28.8, in=180-28.8] (Center-0) to [out=-28.8, in=180+28.8] (Center-4) to [out=28.8, in=90] (5.6, 0) to [out=270, in=0] (5, -3.4) to (-4.8, -3.4) to [out=180, in=180] (-3.8,-1.7);

\end{tikzpicture}
        }
        \caption{\label{occineq} An $n$-crossing projection with two OCC faces can be transformed into an $(n+2)$-crossing projection with one fewer crossing}
\end{figure}

Limiting the $n$-crossing projection of $K$ in Theorem~\ref{accocctheorem} to an $n$-crossing projection with $c_n(K)$ crossings, we achieve the following result:

\begin{corollary}\label{accoccremove}
For odd $n$ and knot $K$ with $c_n(K)>1$, if there exists a minimal $n$-crossing projection with two adjacent or opposite crossing connected faces, then $c_{2n-1}(K)<c_n(K)$.
\end{corollary}

Though the above results are generalized for $n$-crossing projections of any odd $n$, the next part of Section~\ref{sectionclassification} will focus on results specific to the $5$- and $9$- crossing numbers. In addition, since $c_{n+2}(K)\leq c_n(K)$ for all $n$, we know that $c_9(K)\leq c_5(K)$ for any knot $K$. So, we are primarily interested in determining if there exist any knots $K$ such that $c_9(K)=c_5(K)$.

\begin{definition}[Adjoined Bigon]
Define an adjoined bigon to be a polygon with one vertex and two edges, as seen in Figure~\ref{adjoinedbigon}.
\end{definition}

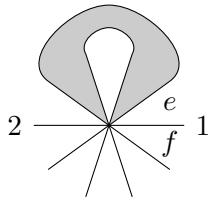
\begin{figure}[!htb]
        \center{
        
\tikzset{invclip/.style={clip,insert path={{[reset cm]
      (-16383.99999pt,-16383.99999pt) rectangle (16383.99999pt,16383.99999pt)
    }}}}
    
\begin{tikzpicture}

\tikzstyle{reverseclip}=[insert path={(current page.north east) --
  (current page.south east) --
  (current page.south west) --
  (current page.north west) --
  (current page.north east)}
]

\coordinate (Center) at (0, 0);

\coordinate (1) at (0:1);
\coordinate (2) at (36:1);
\coordinate (3) at (72:1);
\coordinate (4) at (108:1);
\coordinate (5) at (144:1);
\coordinate (6) at (180:1);
\coordinate (7) at (216:1);
\coordinate (8) at (252:1);
\coordinate (9) at (288:1);
\coordinate (10) at (324:1);

\coordinate (loop1) at (90:1.5);

\begin{scope}
\begin{pgfinterruptboundingbox}
\path[invclip] (Center) to (3) to [out=60,in=0] (90:1.3) to [out=180,in=120] (4) to (Center) -- cycle;
\path[fill=black,opacity=0.2] (Center) to (2) to [out=36,in=0] (90:1.6) to [out=180,in=180-36] (5) to (Center) to (3) to [out=60,in=0] (90:1.3) to [out=180,in=120] (4) to (Center) -- cycle;


\end{pgfinterruptboundingbox}
\end{scope}

\draw (Center) -- (1);
\draw (Center) -- (2);
\draw (Center) -- (3);
\draw (Center) -- (4);
\draw (Center) -- (5);
\draw (Center) -- (6);
\draw (Center) -- (7);
\draw (Center) -- (8);
\draw (Center) -- (9);
\draw (Center) -- (10);

\node at (1) [right = .2mm] {$1$};
\node at (6) [left = .2mm] {$2$};

\node at (2) [below = .8mm] {$e$};
\node at (10) [above = .4mm] {$f$};

\draw (3) to [out=60,in=0] (90:1.3) to [out=180,in=120] (4);
\draw (2) to [out=36,in=0] (90:1.6) to [out=180,in=180-36] (5);

\end{tikzpicture}
        
        }
        \caption{\label{adjoinedbigon} An adjoined bigon as part of a 5-crossing diagram}
\end{figure}

\begin{corollary}\label{adjoinedbigoncor}
If a minimal 5-crossing projection of $K$ for knot $K$ with $c_5(K)>1$ contains an adjoined bigon, then $c_9(K)<c_5(K)$.
\end{corollary}
\begin{proof}
We will use the labeled strands of the adjoined bigon seen in Figure~\ref{adjoinedbigon}. It is easy to see that if both strand 1 and strand 2 connect back to the original crossing, forming a monogon or adjoined bigon, $K$ will either be a link or have $c_5(K)=1$. So, strand 1 or 2 must connect to some other crossing—call it $\beta$. Without a loss of generality, say strand 1 connects to $\beta$. This means the faces $e$ and $f$ are ACC via $\beta$, and $c_9(K)<c_5(K)$.
\end{proof}

\begin{corollary}\label{atmost2}
If a minimal 5-crossing projection of $K$ for knot $K$ with $c_5(K)>1$ contains a crossing with three or more monogons, then $c_9(K)<c_5(K)$.
\end{corollary}
\begin{proof}
A crossing of $K$ can't have 5 monogons as then $c_5(K)=1$. If a crossing of $K$ has 4 monogons, two of the monogons must be consecutive. This gives the crossing seen in Figure~\ref{adjacentmonogons}. Faces $e$ and $f$ are ACC via $\beta$, and we are done. The only crossing with at least three monogons and no two consecutive monogons is seen in Figure~\ref{3nonadjacentmonogons}. Again, faces $e$ and $f$ are ACC via $\beta$, which means $c_9(K)<c_5(K)$.
\end{proof}

\begin{figure}[!htb]
        \center{

\begin{tikzpicture}
\coordinate (Center-1) at (0-2, 0);

\coordinate (1-1) at (0.809-2, 0.588);
\coordinate (2-1) at (0.309-2, 0.951);
\coordinate (3-1) at (-0.309-2, 0.951);
\coordinate (4-1) at (-0.809-2,0.588);
\coordinate (5-1) at (-1-2,0);
\coordinate (6-1) at (-0.809-2,-.588);
\coordinate (7-1) at (-0.309-2,-.951);
\coordinate (8-1) at (0.309-2,-.951);
\coordinate (9-1) at (0.809-2,-.588);
\coordinate (10-1) at (1-2,0);

\coordinate (Center-2) at (0+2, 0);

\coordinate (1-2) at (0.809+2, 0.588);
\coordinate (2-2) at (0.309+2, 0.951);
\coordinate (3-2) at (-0.309+2, 0.951);
\coordinate (4-2) at (-0.809+2,0.588);
\coordinate (5-2) at (-1+2,0);
\coordinate (6-2) at (-0.809+2,-.588);
\coordinate (7-2) at (-0.309+2,-.951);
\coordinate (8-2) at (0.309+2,-.951);
\coordinate (9-2) at (0.809+2,-.588);
\coordinate (10-2) at (1+2,0);

\draw (Center-1) -- (1-1);
\draw (Center-1) -- (2-1);
\draw (Center-1) -- (3-1);
\draw (Center-1) -- (4-1);
\draw (Center-1) -- (5-1);
\draw (Center-1) -- (6-1);
\draw (Center-1) -- (7-1);
\draw (Center-1) -- (8-1);
\draw (Center-1) -- (9-1);
\draw (Center-1) -- (10-1);

\draw (Center-2) -- (1-2);
\draw (Center-2) -- (2-2);
\draw (Center-2) -- (3-2);
\draw (Center-2) -- (4-2);
\draw (Center-2) -- (5-2);
\draw (Center-2) -- (6-2);
\draw (Center-2) -- (7-2);
\draw (Center-2) -- (8-2);
\draw (Center-2) -- (9-2);
\draw (Center-2) -- (10-2);


\node at (Center-1) [shift={(-0.35,0.5)}] {$\alpha$};
\node at (Center-2) [shift={(0.35,0.5)}] {$\beta$};

\node at (0, 0.7) {$f$};
\node at (0, 1.2) {$e$};

\draw (10-2) to [out=0,in=36] (1-2);
\draw (2-2) to [out=72,in=108] (3-2);
\draw (1-1) to [out=36, in=180] (0, 1) to [out=0, in=144] (4-2);
\draw [dashed, red] (0.475 - 2, 0.1545) to [out=18, in=162] (Center-2) to [out=-18, in=180+90] (3.4, 0) to [out=90, in=-30] (0.309*1.5+2, 0.951*1.5) to [out=150, in=48] (0.2938-2, 0.4045);

\end{tikzpicture}
}
\caption{\label{adjacentmonogons} Crossing with 2 consecutive monogons}
\end{figure}
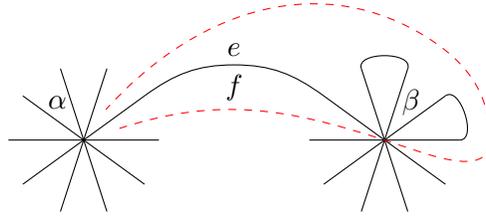

\begin{figure}[!htb]
        \center{\begin{tikzpicture}
\coordinate (Center-1) at (0-2, 0);

\coordinate (1-1) at (0.809-2, 0.588);
\coordinate (2-1) at (0.309-2, 0.951);
\coordinate (3-1) at (-0.309-2, 0.951);
\coordinate (4-1) at (-0.809-2,0.588);
\coordinate (5-1) at (-1-2,0);
\coordinate (6-1) at (-0.809-2,-.588);
\coordinate (7-1) at (-0.309-2,-.951);
\coordinate (8-1) at (0.309-2,-.951);
\coordinate (9-1) at (0.809-2,-.588);
\coordinate (10-1) at (1-2,0);

\coordinate (Center-2) at (0+2, 0);

\coordinate (1-2) at (0.809+2, 0.588);
\coordinate (2-2) at (0.309+2, 0.951);
\coordinate (3-2) at (-0.309+2, 0.951);
\coordinate (4-2) at (-0.809+2,0.588);
\coordinate (5-2) at (-1+2,0);
\coordinate (6-2) at (-0.809+2,-.588);
\coordinate (7-2) at (-0.309+2,-.951);
\coordinate (8-2) at (0.309+2,-.951);
\coordinate (9-2) at (0.809+2,-.588);
\coordinate (10-2) at (1+2,0);

\draw (Center-1) -- (1-1);
\draw (Center-1) -- (2-1);
\draw (Center-1) -- (3-1);
\draw (Center-1) -- (4-1);
\draw (Center-1) -- (5-1);
\draw (Center-1) -- (6-1);
\draw (Center-1) -- (7-1);
\draw (Center-1) -- (8-1);
\draw (Center-1) -- (9-1);
\draw (Center-1) -- (10-1);

\draw (Center-2) -- (1-2);
\draw (Center-2) -- (2-2);
\draw (Center-2) -- (3-2);
\draw (Center-2) -- (4-2);
\draw (Center-2) -- (5-2);
\draw (Center-2) -- (6-2);
\draw (Center-2) -- (7-2);
\draw (Center-2) -- (8-2);
\draw (Center-2) -- (9-2);
\draw (Center-2) -- (10-2);

\node at (Center-1) [shift={(-0.35,0.5)}] {$\alpha$};
\node at (Center-2) [shift={(0.6,0.2)}] {$\beta$};

\node at (0, 0.7) {$f$};
\node at (0, 1.2) {$e$};

\draw (9-2) to [out=-36,in=0] (10-2);
\draw (2-2) to [out=72,in=108] (3-2);
\draw (5-2) to [out=180,in=216] (6-2);

\draw (1-1) to [out=36, in=180] (0, 1) to [out=0, in=144] (4-2);
\draw [dashed, red] (0.475 - 2, 0.1545) to [out=18, in=180] (1, -0.8) to [out=0, in=234] (Center-2) to [out=54, in=-30] (2.5, 1.5) to [out=180 - 30, in=48] (0.2938-2, 0.4045);

\end{tikzpicture}}
        \caption{\label{3nonadjacentmonogons} Crossing with 3 non-consecutive monogons}
\end{figure}
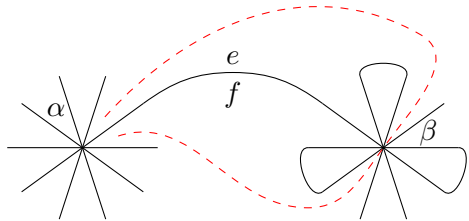

\begin{definition}[Almost Opposite]
Two monogons are almost opposite if they share the same vertex and if there are exactly two strands in between them.
\end{definition}

These 3 corollaries give the following theorem:

\begin{theorem}\label{endsection1proof}
Let $K$ be a knot with $c_5(K)>1$. Then $c_5(K)= c_9(K)$ can only hold if the following 3 conditions hold in every minimal 5-crossing projection of $K$: (1) every crossing has either no monogons, one monogon, or two almost opposite monogons; (2) no crossing has an adjoined bigon; and (3) no two faces are ACC or OCC.
\end{theorem}
\begin{proof}
Parts (2) and (3) have already been proven in Corollary~\ref{adjoinedbigoncor} and Corollary~\ref{accoccremove}, respectively. Corollary~\ref{atmost2} gives that $c_9(K)=c_5(K)$ only if all crossings have at most 2 monogons. Consider the case where a crossing has 2 monogons. The monogons can't be opposite as $K$ would be a link. If the monogons are consecutive or separated by one strand, this gives crossings similar to the ones in Figure~\ref{adjacentmonogons} and Figure~\ref{3nonadjacentmonogons}. But then face $e$ and $f$ are ACC via $\beta$, and $c_9(K)<c_5(K)$. Then, every crossing in the minimal 5-crossing diagram of $K$ must be a crossing with no monogons, a crossing with one monogon, or a crossing with two almost opposite monogons. So, the claim holds.

\end{proof}

\begin{corollary}\label{bruteforce}
    If knot $K$ has $c_5(K)=2$, then $c_9(K)=1$.
\end{corollary}
\begin{proof}
    Consider a $5$-crossing diagram of $K$ with 2 crossings. By Theorem~\ref{endsection1proof}, each crossing can contain at most two monogons. With only two crossings, both crossings must have the same number of monogons. Disregarding crossing information, there are 15 total knot diagrams with 2 crossings and where each crossing has either no monogons, one monogon, or two almost opposite monogons. In each of the fifteen cases, one can check that there exists two ACC or OCC faces. Then by Theorem~\ref{endsection1proof}, $c_9(K)=1$. 
\end{proof}

Corollary~\ref{bruteforce} allows us to obtain previously unknown values of $c_9(K)$ for a few knots. Multi-crossing tabulation in \cite{generaltabulation} obtained the fifth crossing number but not the ninth crossing number of the knots in Table~\ref{tabtable}.

\begin{table}[! h]
\centering
\caption{Knots with $c_5(K)=2$ and previously unknown 9-crossing number}\label{tabtable}
\vspace{4mm}
\begin{tabular}{l|l|l}
\textbf{Knot} & \textbf{$c_5(K)$} & \textbf{$c_9(K)$} \\ \hline

$8_5$                   & 2    & 1                                   \\ \hline
$8_{17}$                   & 2     & 1                            \\ \hline
$8_{18}$                   & 2        & 1                            \\ \hline
$4_1\#4_1$                   & 2         & 1                          \\ \hline
$9_{47}$                   & 2              & 1                       \\ \hline
$9_{49}$                   & 2                 & 1                     \\ 
\end{tabular}
\end{table}

\section{A 9-crossing number inequality}\label{section9crossing}

It has been shown in \cite{tabulation} that the trivial, trefoil, and figure-eight knots are the only knots $K$ such that $c_3(K)\leq 2$. So, for all results in the following section, we will restrict all knots $K$ to be knots that are not the trivial, trefoil, or figure-eight knot. For these remaining knots, we will create a 9-crossing diagram with $c_3(K)-2$ crossings by first constructing a 5-crossing projection with $c_3(K)-1$ crossings. Adams  proved this first step is always possible \cite{tripmoves}. We will instead show that all such knots contain a 5-crossing projection with $c_3(K)-1$ crossings that also contains at least one crossing of the type described in Section~\ref{sectionclassification}. Then, Theorem~\ref{endsection1proof} gives that $c_9(K)\leq c_3(K)-2$.\\

We will prove that $c_9(K)\leq c_3(K)-2$ in three main lemmas. Lemma~\ref{twomonogons} considers the case where a minimal triple crossing diagram of $K$ contains two or more monogons. Lemma~\ref{onemonogon} considers the case where a minimal triple crossing diagram of $K$ contains one monogon. Lastly, Lemma~\ref{zeromonogon} considers the case where a minimal triple crossing diagram of $K$ contains zero monogons.

\begin{lemma}\label{twomonogons}
Let $K$ be a knot that is not the trivial, trefoil, or figure-eight knot. If a minimal 3-crossing diagram of $K$ contains at least two monogons, then $c_9(K)\leq c_3(K)-2$.
\end{lemma}

\begin{proof}

For the minimal 3-crossing diagram of $K$ with two monogons, take a crossing $\alpha$ that contains a monogon. Crossing $\alpha$ cannot be that of Figure~\ref{oppositemonogons} because it is a link, and it can't be that of Figure~\ref{oneoffmonogon} as, no matter how the crossing information is filled in, the crossing can be removed by the 3-crossing 1-move described in \cite{tripmoves}. This leaves the crossing seen in Figure~\ref{onemonogoncrossing}. Also, since this means $\alpha$ only has one monogon, there exists at least one other crossing with a monogon. Let $\beta$ be such a crossing.\\

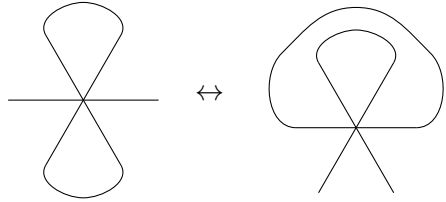
\begin{figure}[!htb]
        \center{\begin{tabular}{lll}
     \raisebox{-.5\height}{
 \begin{tikzpicture}
 \coordinate (Center) at (0, 0);
 
 \coordinate (1) at (0:1);
 \coordinate (2) at (60:1);
 \coordinate (3) at (120:1);
 \coordinate (4) at (180:1);
 \coordinate (5) at (240:1);
 \coordinate (6) at (300:1);
 
 \draw (Center) -- (1);
 \draw (Center) -- (2);
 \draw (Center) -- (3);
 \draw (Center) -- (4);
 \draw (Center) -- (5);
 \draw (Center) -- (6);
 
 \draw (2) to [out=60,in=0] (90:1.3) to [out=180,in=120] (3);
 \draw (5) to [out=180+60,in=180] (270:1.3) to [out=0,in=120+180] (6);

 \end{tikzpicture}
 
 } & $\leftrightarrow$  & \raisebox{-.5\height}{
 
 \begin{tikzpicture}
 \coordinate (Center) at (0, 0);
 
 \coordinate (1) at (0:0.8);
 \coordinate (2) at (60:1);
 \coordinate (3) at (120:1);
 \coordinate (4) at (180:0.8);
 \coordinate (5) at (240:1);
 \coordinate (6) at (300:1);
 \coordinate (loop1) at (90:1.5);
 
 \draw (Center) -- (1);
 \draw (Center) -- (2);
 \draw (Center) -- (3);
 \draw (Center) -- (4);
 \draw (Center) -- (5);
 \draw (Center) -- (6);
 
 \draw (2) to [out=60,in=0] (90:1.3) to [out=180,in=120] (3);
 \draw (1) to [out=0,in=-45] (45:1.4) to [out=-45+180,in=0] (90:1.6) to [out=180,in=45] (135:1.4) to [out=45+180,in=180] (4);

 \end{tikzpicture}}\\
 \end{tabular}}
         \caption{\label{oppositemonogons} Crossings that are part of a link}
 \end{figure}
 
 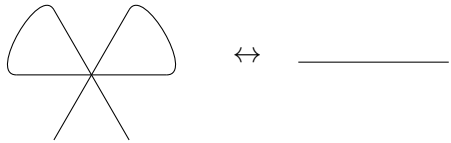
\begin{figure}[!htb]
 
 \center{\begin{tabular}{lll}
     \raisebox{-.5\height}{
 \begin{tikzpicture}
 \coordinate (Center) at (0, 0);
 
 \coordinate (1) at (0:1);
 \coordinate (2) at (60:1);
 \coordinate (3) at (120:1);
 \coordinate (4) at (180:1);
 \coordinate (5) at (240:1);
 \coordinate (6) at (300:1);
 
 \draw (Center) -- (1);
 \draw (Center) -- (2);
 \draw (Center) -- (3);
 \draw (Center) -- (4);
 \draw (Center) -- (5);
 \draw (Center) -- (6);
 
 \draw (1) to [out=0,in=60] (2);
 
 \draw (3) to [out=120,in=180] (4);

 \end{tikzpicture}
 
 } & $\leftrightarrow$  & \raisebox{-.5\height}{
 
 \begin{tikzpicture}
 \coordinate (Center) at (0, 0);
 
 \coordinate (1) at (0:1);
 \coordinate (2) at (60:1);
 \coordinate (3) at (120:1);
 \coordinate (4) at (180:1);
 \coordinate (5) at (240:1);
 \coordinate (6) at (300:1);
 
 \draw (1) to (4);

 \end{tikzpicture}}\\
 \end{tabular}}

         \caption{\label{oneoffmonogon} Crossing that can be removed with a 3-crossing 1-move}
 \end{figure}

\begin{figure}[!htb]
        \center{
\begin{tikzpicture}
\coordinate (Center) at (0, 0);

\coordinate[label=above:2] (1) at (0:1);

\coordinate (2) at (60:1);
\coordinate (3) at (120:1);
\coordinate[label=above:1] (4) at (180:1);
\coordinate (5) at (240:1);
\coordinate (6) at (300:1);

\draw (Center) -- (1);
\draw (Center) -- (2);
\draw (Center) -- (3);
\draw (Center) -- (4);
\draw (Center) -- (5);
\draw (Center) -- (6);

\draw (2) to [out=60,in=0] (90:1.3) to [out=180,in=120] (3);

\end{tikzpicture}}
        \caption{\label{onemonogoncrossing} Diagram of $\alpha$}
\end{figure}
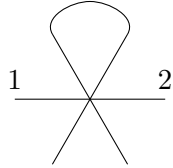

Call the crossings that connect to $\alpha$ by the 1 and 2 strands its adjacent crossings. $\beta$ has similar adjacent crossings. Crossings $\alpha$ and $\beta$ can't both have their adjacent crossings be each other as that would leave us with a link, as seen in Figure~\ref{betasame}. So, without a loss of generality, let $\alpha$ be the crossing that contains at least one adjacent crossing that is not the other crossing. So, take an adjacent crossing to $\alpha$ that is not $\beta$. Call it $\gamma$. There now exists a 3-crossing ACC via $\gamma$. We can wrap strands of $\gamma$ around $\alpha$ to eliminate the former crossing while converting the latter to a 5-crossing, as described in \cite{tripmoves}.\\

\begin{figure}[!htb]
        \center{\begin{tikzpicture}[scale=0.7]
\coordinate (Center-1) at (0-2, 0);
\coordinate (1-1) at (0.86602540378-2, 0.5);
\coordinate (2-1) at (0-2,1);
\coordinate (3-1) at (-0.86602540378-2, 0.5);
\coordinate (4-1) at (-0.86602540378-2, -0.5);
\coordinate (5-1) at (0-2,-1);
\coordinate (6-1) at (0.86602540378-2, -0.5);

\coordinate (Center-2) at (0+2, 0);
\coordinate (1-2) at (0.86602540378+2, 0.5);
\coordinate (2-2) at (0+2,1);
\coordinate (3-2) at (-0.86602540378+2, 0.5);
\coordinate (4-2) at (-0.86602540378+2, -0.5);
\coordinate (5-2) at (0+2,-1);
\coordinate (6-2) at (0.86602540378+2, -0.5);

\draw (Center-1) -- (1-1);
\draw (Center-1) -- (2-1);
\draw (Center-1) -- (3-1);
\draw (Center-1) -- (4-1);
\draw (Center-1) -- (5-1);
\draw (Center-1) -- (6-1);

\draw (Center-2) -- (1-2);
\draw (Center-2) -- (2-2);
\draw (Center-2) -- (3-2);
\draw (Center-2) -- (4-2);
\draw (Center-2) -- (5-2);
\draw (Center-2) -- (6-2);

\draw (2-1) to [out=90,in=180] (90:1.8) to [out=0,in=90] (2-2);
\draw (5-1) to [out=-90,in=-180] (270:1.8) to [out=0,in=-90] (5-2);

\end{tikzpicture}}
        \caption{\label{betasame} If $\beta$ and $\alpha$'s 1 and 2 stands are joined, $K$ is a link}
\end{figure}
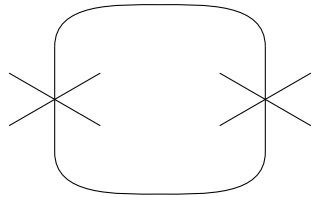

Let us also transform all other 3-crossings to 5-crossings using the move shown in Figure~\ref{increasebytwofirsttime}. This adds two monogons to all these crossings, leaving $\beta$ with at least three monogons. Yet this is one of the crossings detailed in Theorem~\ref{endsection1proof}, meaning there exists a 9-crossing projection of $K$ with one fewer crossing than this 5-crossing diagram. That means $c_9(K)\leq c_3(K)-2$.

\end{proof}

\begin{lemma}\label{algorithmsolves}
Let $K$ be a knot that is not the trivial, trefoil, or figure-eight knot. If a minimal 3-crossing diagram of $K$ (1) contains a crossing $\alpha$ with a monogon and (2) there exist two faces that are crossing connected via a crossing that is not $\alpha$ where (3) the crossing segment does not intersect $\alpha$, then $c_9(K)\leq c_3(K)-2$.
\end{lemma}
\begin{proof}
By the methods used in Theorem~\ref{accocctheorem}, we can transform the 3-crossings around the crossing connection into a 5-crossings, eliminating one crossing in the process. Crossing $\alpha$ will remain untouched, meaning we can apply the transformation seen in Figure~\ref{increasebytwofirsttime} to $\alpha$ and all other untouched crossings. This gives a 5-crossing diagram with $c_3(K)-1$ crossings. Crossing $\alpha$ now contains three monogons, meaning $c_9(K)\leq c_3(K)-2$ by Theorem~\ref{endsection1proof}.
\end{proof}

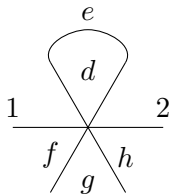
\begin{figure}[!htb]
        \center{\begin{tikzpicture}
\coordinate (Center) at (0, 0);
\node at (Center) [above = 5mm of Center] {$d$};
\node at (Center) [below = 5mm of Center] {$g$};

\node at (Center) [above = 13 mm of Center] {$e$};

\coordinate[label=above:2] (1) at (0:1);

\coordinate (2) at (60:1);
\coordinate (3) at (120:1);
\coordinate[label=above:1] (4) at (180:1);
\coordinate (5) at (240:1);
\coordinate (6) at (300:1);

\node at (5) [above = 2mm of 5] {$f$};
\node at (6) [above = 2mm of 6] {$h$};

\draw (Center) -- (1);
\draw (Center) -- (2);
\draw (Center) -- (3);
\draw (Center) -- (4);
\draw (Center) -- (5);
\draw (Center) -- (6);

\draw (2) to [out=60,in=0] (90:1.3) to [out=180,in=120] (3);

\end{tikzpicture}}
        \caption{\label{alphasfaces} Labeling of the faces adjacent to $\alpha$}
\end{figure}

\begin{lemma}\label{onemonogon}
Let $K$ be a knot that is not the trivial, trefoil, or figure-eight knot. If a minimal 3-crossing diagram of $K$ contains one monogon, then $c_9(K)\leq c_3(K)-2$.
\end{lemma}
\begin{proof}
Consider a 3-crossing projection of $K$ with $c_3(K)$ crossings and one monogon. Let $\alpha$ be the crossing containing the monogon, and label its adjacent faces as seen in Figure~\ref{alphasfaces}. We will examine three cases, according to the number of edges of $f$ and $g$. For each case, we will present an algorithm to find two faces that are crossing connected via a crossing that is not $\alpha$ where the crossing segment does not intersect $\alpha$. Finding such a crossing connection implies $c_9(K)\leq c_3(K)-2$ for the given knot $K$ by Lemma~\ref{algorithmsolves}. We now describe the appropriate algorithm for each case.\\

\textbf{Case 1:} \textit{Faces $f$ and $g$ are both bigons.} Our algorithm will be as follows: Start at crossing $\beta$ and travel to crossing $\gamma$, as shown in Figure~\ref{fgbigon}. Whenever we reach a crossing, we must travel to the exact opposite side of the crossing. Whenever we reach a face, arbitrarily choose one of the face's vertices to travel to next. However, ensure that the vertex we travel to is not the one we just came from, is not $\alpha$, and does not require we cross our path. The algorithm terminates when we hit some crossing for the second time. Since there are a finite number of crossings, our algorithm must either fail or terminate.\\

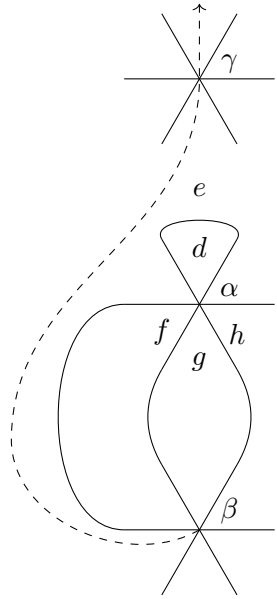
\begin{figure}[!htb]
        \center{\begin{tikzpicture}
\coordinate (Center-1) at (0, 0-3);

\coordinate (1-1) at (1,0-3);
\coordinate (2-1) at (0.5, 0.86602540378-3);
\coordinate (3-1) at (-0.5, 0.86602540378-3);
\coordinate (4-1) at (-1,0-3);
\coordinate (5-1) at (-0.5, -0.86602540378-3);
\coordinate (6-1) at (0.5, -0.86602540378-3);

\coordinate (Center-2) at (0, 0);
\coordinate (1-2) at (1,0);
\coordinate (2-2) at (0.5, 0.86602540378);
\coordinate (3-2) at (-0.5, 0.86602540378);
\coordinate (4-2) at (-1,0);
\coordinate (5-2) at (-0.5, -0.86602540378);
\coordinate (6-2) at (0.5, -0.86602540378);

\coordinate (Center-3) at (0, 0+3);
\coordinate (1-3) at (1,0+3);
\coordinate (2-3) at (0.5, 0.86602540378+3);
\coordinate (3-3) at (-0.5, 0.86602540378+3);
\coordinate (4-3) at (-1,0+3);
\coordinate (5-3) at (-0.5, -0.86602540378+3);
\coordinate (6-3) at (0.5, -0.86602540378+3);

\draw (Center-1) -- (1-1);
\draw (Center-1) -- (2-1);
\draw (Center-1) -- (3-1);
\draw (Center-1) -- (4-1);
\draw (Center-1) -- (5-1);
\draw (Center-1) -- (6-1);

\draw (Center-2) -- (1-2);
\draw (Center-2) -- (2-2);
\draw (Center-2) -- (3-2);
\draw (Center-2) -- (4-2);
\draw (Center-2) -- (5-2);
\draw (Center-2) -- (6-2);

\draw (Center-3) -- (1-3);
\draw (Center-3) -- (2-3);
\draw (Center-3) -- (3-3);
\draw (Center-3) -- (4-3);
\draw (Center-3) -- (5-3);
\draw (Center-3) -- (6-3);

\node at (Center-2) [above = 5mm] {$d$};
\node at (Center-2) [below = 5mm] {$g$};
\node at (Center-2) [above = 13 mm] {$e$};
\node at (5-2) [above = 2mm] {$f$};
\node at (6-2) [above = 2mm] {$h$};

\node at (Center-1) [shift={(0.4,0.25)}] {$\beta$};
\node at (Center-2) [shift={(0.4,0.2)}] {$\alpha$};
\node at (Center-3) [shift={(0.4,0.2)}] {$\gamma$};

\draw (2-1) to [out=60, in=180 + 120] (6-2);
\draw (3-1) to [out=120, in=180+60] (5-2);
\draw (4-1) to [out=180, in=180] (4-2);
\draw (2-2) to [out=60, in=120](3-2);
\draw[dashed,->] (Center-1) to [out=210, in=270] (-2.5, -1.8) to [out=90, in = 270](Center-3) to (0, 4);

\end{tikzpicture}}
        \caption{\label{fgbigon} Face $f$ and $g$ are both bigons}
\end{figure}

We will show the algorithm never fails. It is easy to see that the algorithm is never forced to cross its path and is only forced to intersect the same crossing twice in a row if it arrives at a monogon. The only monogon is $d$, which the path can only reach by first traversing $\alpha$. So, showing that the path never needs to cross $\alpha$ is enough to prove the algorithm never fails. Now, to intersect $\alpha$, the path must first reach $e$, $f$, $g$, or $h$. But if the path reaches any of these faces, the path will instead terminate by intersecting $\beta$ for the second time. Therefore, the path must always terminate.\\

Now, after our algorithm terminates, consider the crossing we traversed twice on our path—call this crossing $\delta$. Due to the restrictions of the algorithm, the part of our path after first traversing $\delta$ and before traversing $\delta$ a second time is a crossing segment that doesn't intersect $\alpha$. This makes two of the faces that touch $\delta$ crossing connected via $\delta$. We are done by Lemma~\ref{algorithmsolves}.\\

\textbf{Case 2:} \textit{Face $g$ is a bigon while face $f$ isn't.} We use the same algorithm used in Case~1, again starting at crossing $\beta$ and traveling to crossing $\gamma$, as shown in Figure~\ref{gbigon}. The algorithm cannot fail as the path must reach $e$, $f$, $g$, or $h$ first. But by reaching $f$, $g$, or $h$ the path will instead terminate by intersecting $\beta$ for the second time. And by reaching $e$, the path will stop by traversing $\gamma$ for the second time. So the algorithm terminates, and, for the same reason described in Case 1, we are done.\\

\begin{figure}[!htb]
        \center{\begin{tikzpicture}
\coordinate (Center-1) at (0, 0-3);

\coordinate (1-1) at (1,0-3);
\coordinate (2-1) at (0.5, 0.86602540378-3);
\coordinate (3-1) at (-0.5, 0.86602540378-3);
\coordinate (4-1) at (-1,0-3);
\coordinate (5-1) at (-0.5, -0.86602540378-3);
\coordinate (6-1) at (0.5, -0.86602540378-3);

\coordinate (Center-2) at (0, 0);
\coordinate (1-2) at (1,0);
\coordinate (2-2) at (0.5, 0.86602540378);
\coordinate (3-2) at (-0.5, 0.86602540378);
\coordinate (4-2) at (-1,0);
\coordinate (5-2) at (-0.5, -0.86602540378);
\coordinate (6-2) at (0.5, -0.86602540378);

\coordinate (Center-3) at (0-3, 0);
\coordinate (1-3) at (1-3,0);
\coordinate (2-3) at (0.5-3, 0.86602540378);
\coordinate (3-3) at (-0.5-3, 0.86602540378);
\coordinate (4-3) at (-1-3,0);
\coordinate (5-3) at (-0.5-3, -0.86602540378);
\coordinate (6-3) at (0.5-3, -0.86602540378);

\draw (Center-1) -- (1-1);
\draw (Center-1) -- (2-1);
\draw (Center-1) -- (3-1);
\draw (Center-1) -- (4-1);
\draw (Center-1) -- (5-1);
\draw (Center-1) -- (6-1);

\draw (Center-2) -- (1-2);
\draw (Center-2) -- (2-2);
\draw (Center-2) -- (3-2);
\draw (Center-2) -- (4-2);
\draw (Center-2) -- (5-2);
\draw (Center-2) -- (6-2);

\draw (Center-3) -- (1-3);
\draw (Center-3) -- (2-3);
\draw (Center-3) -- (3-3);
\draw (Center-3) -- (4-3);
\draw (Center-3) -- (5-3);
\draw (Center-3) -- (6-3);

\node at (Center-2) [above = 5mm] {$d$};
\node at (Center-2) [below = 5mm] {$g$};
\node at (Center-2) [above = 13 mm] {$e$};
\node at (5-2) [above = 2mm] {$f$};
\node at (6-2) [above = 2mm] {$h$};

\node at (Center-1) [shift={(0.4,0.25)}] {$\beta$};
\node at (Center-2) [shift={(0.4,0.2)}] {$\alpha$};
\node at (Center-3) [shift={(0.4,0.2)}] {$\gamma$};

\draw (2-1) to [out=60, in=180 + 120] (6-2);
\draw (3-1) to [out=120, in=180+60] (5-2);
\draw (1-3) to [out=0, in=180] (4-2);
\draw (2-2) to [out=60, in=120](3-2);
\draw[dashed,->] (Center-1) to [out=150, in=-30] (Center-3) to (-0.866 - 3, 0.5);

\end{tikzpicture}}
        \caption{\label{gbigon} Face $g$ is a bigon while face $f$ isn't}
\end{figure}
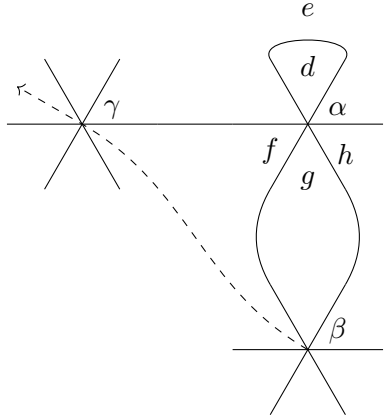

\textbf{Case 3:} \textit{Face $g$ isn't a bigon.} First, if face $e$ has 3 sides, we can perform the move seen in Figure \ref{egswitch}. This makes face $g$ a bigon. By Case~1 and Case~2, we are then done. Face $e$ can't have two sides as that makes $K$ a link, and face $e$ can't have one side because $d$ is the only monogon. So, the only remaining situation is the case where $e$ has more than three sides. We can now perform an algorithm similar to the one used in Case~1. However, we will instead start our algorithm at crossing $\delta$ and initially head towards crossing $\eps$, as seen in Figure~\ref{nonebigon}. The rules for the path are the same as the rules in Case~1. Once again, our algorithm terminates when the path intersects a crossing for the second time. Additionally, since there are a finite number of crossings, it must either terminate or fail.\\

The algorithm cannot fail as the path must reach $e$, $f$, $g$, or $h$ before crossing $\alpha$. However, by reaching $f$, $h$, or $e$ the path will instead terminate by intersecting $\delta$ or $\eps$ for the second time. Additionally, by reaching $g$, the path always has another crossing it can traverse instead of $\alpha$ as $g$ borders at least two vertices that aren't $\alpha$. For instance, if the path comes from $\beta$, it can next cross $\gamma$. So the algorithm terminates, and we are done by the same reason as in Case 1.\\ 

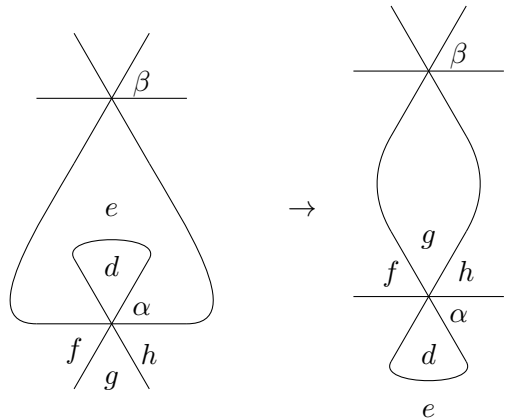
\begin{figure}[!htb]
        \center{\begin{tabular}{lll}
    \raisebox{-.5\height}{\begin{tikzpicture}

\coordinate (Center-2) at (0, 0);
\coordinate (1-2) at (1,0);
\coordinate (2-2) at (0.5, 0.86602540378);
\coordinate (3-2) at (-0.5, 0.86602540378);
\coordinate (4-2) at (-1,0);
\coordinate (5-2) at (-0.5, -0.86602540378);
\coordinate (6-2) at (0.5, -0.86602540378);

\coordinate (Center-3) at (0, 0+3);
\coordinate (1-3) at (1,0+3);
\coordinate (2-3) at (0.5, 0.86602540378+3);
\coordinate (3-3) at (-0.5, 0.86602540378+3);
\coordinate (4-3) at (-1,0+3);
\coordinate (5-3) at (-0.5, -0.86602540378+3);
\coordinate (6-3) at (0.5, -0.86602540378+3);

\draw (Center-2) -- (1-2);
\draw (Center-2) -- (2-2);
\draw (Center-2) -- (3-2);
\draw (Center-2) -- (4-2);
\draw (Center-2) -- (5-2);
\draw (Center-2) -- (6-2);

\draw (Center-3) -- (1-3);
\draw (Center-3) -- (2-3);
\draw (Center-3) -- (3-3);
\draw (Center-3) -- (4-3);
\draw (Center-3) -- (5-3);
\draw (Center-3) -- (6-3);

\node at (Center-2) [above = 5mm] {$d$};
\node at (Center-2) [below = 5mm] {$g$};
\node at (Center-2) [above = 13 mm] {$e$};
\node at (5-2) [above = 2mm] {$f$};
\node at (6-2) [above = 2mm] {$h$};

\node at (Center-2) [shift={(0.4,0.2)}] {$\alpha$};
\node at (Center-3) [shift={(0.4,0.2)}] {$\beta$};

\draw (4-2) to [out=180, in=240] (5-3);
\draw (1-2) to [out=0, in=300] (6-3);
\draw (2-2) to [out=60, in=120](3-2);

\end{tikzpicture}

} & $\rightarrow$  & \raisebox{-.5\height}{
\begin{tikzpicture}

\coordinate (Center-2) at (0, 0);
\coordinate (1-2) at (1,0);
\coordinate (2-2) at (0.5, 0.86602540378);
\coordinate (3-2) at (-0.5, 0.86602540378);
\coordinate (4-2) at (-1,0);
\coordinate (5-2) at (-0.5, -0.86602540378);
\coordinate (6-2) at (0.5, -0.86602540378);

\coordinate (Center-3) at (0, 0+3);
\coordinate (1-3) at (1,0+3);
\coordinate (2-3) at (0.5, 0.86602540378+3);
\coordinate (3-3) at (-0.5, 0.86602540378+3);
\coordinate (4-3) at (-1,0+3);
\coordinate (5-3) at (-0.5, -0.86602540378+3);
\coordinate (6-3) at (0.5, -0.86602540378+3);

\draw (Center-2) -- (1-2);
\draw (Center-2) -- (2-2);
\draw (Center-2) -- (3-2);
\draw (Center-2) -- (4-2);
\draw (Center-2) -- (5-2);
\draw (Center-2) -- (6-2);

\draw (Center-3) -- (1-3);
\draw (Center-3) -- (2-3);
\draw (Center-3) -- (3-3);
\draw (Center-3) -- (4-3);
\draw (Center-3) -- (5-3);
\draw (Center-3) -- (6-3);

\node at (Center-2) [above = 5mm] {$g$};
\node at (Center-2) [below = 5mm] {$d$};
\node at (Center-2) [below = 13 mm] {$e$};
\node at (3-2) [below = 2.5mm] {$f$};
\node at (2-2) [below = 2.5mm] {$h$};

\node at (Center-2) [shift={(0.4,-0.25)}] {$\alpha$};
\node at (Center-3) [shift={(0.4,0.2)}] {$\beta$};

\draw (3-2) to [out=180-60, in=240] (5-3);
\draw (2-2) to [out=60, in=300] (6-3);
\draw (5-2) to [out=60+180, in=120+180](6-2);

\end{tikzpicture}}
\end{tabular}}
        \caption{\label{egswitch} Face $e$ is a triangle}
\end{figure}

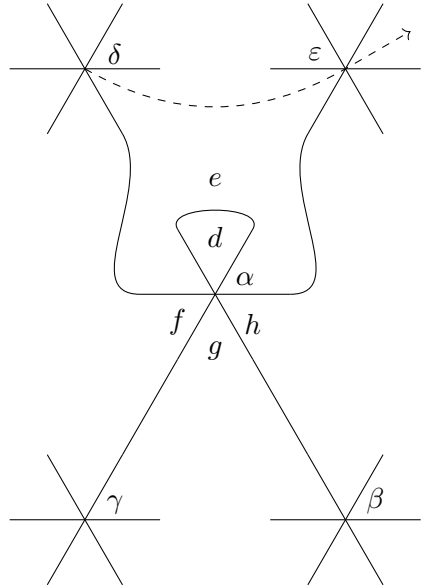
\begin{figure}[!htb]
        \center{
\begin{tikzpicture}
\coordinate (Center-1) at (0+1.73205081, 0-3);

\coordinate (1-1) at (1+1.73205081,0-3);
\coordinate (2-1) at (0.5+1.73205081, 0.86602540378-3);
\coordinate (3-1) at (-0.5+1.73205081, 0.86602540378-3);
\coordinate (4-1) at (-1+1.73205081,0-3);
\coordinate (5-1) at (-0.5+1.73205081, -0.86602540378-3);
\coordinate (6-1) at (0.5+1.73205081, -0.86602540378-3);

\coordinate (Center-2) at (0, 0);
\coordinate (1-2) at (1,0);
\coordinate (2-2) at (0.5, 0.86602540378);
\coordinate (3-2) at (-0.5, 0.86602540378);
\coordinate (4-2) at (-1,0);
\coordinate (5-2) at (-0.5, -0.86602540378);
\coordinate (6-2) at (0.5, -0.86602540378);

\coordinate (Center-3) at (0-1.73205081, 0-3);
\coordinate (1-3) at (1-1.73205081,0-3);
\coordinate (2-3) at (0.5-1.73205081, 0.86602540378-3);
\coordinate (3-3) at (-0.5-1.73205081, 0.86602540378-3);
\coordinate (4-3) at (-1-1.73205081,0-3);
\coordinate (5-3) at (-0.5-1.73205081, -0.86602540378-3);
\coordinate (6-3) at (0.5-1.73205081, -0.86602540378-3);

\coordinate (Center-4) at (0+1.73205081, 0+3);
\coordinate (1-4) at (1+1.73205081,0+3);
\coordinate (2-4) at (0.5+1.73205081, 0.86602540378+3);
\coordinate (3-4) at (-0.5+1.73205081, 0.86602540378+3);
\coordinate (4-4) at (-1+1.73205081,0+3);
\coordinate (5-4) at (-0.5+1.73205081, -0.86602540378+3);
\coordinate (6-4) at (0.5+1.73205081, -0.86602540378+3);

\coordinate (Center-5) at (0-1.73205081, 0+3);
\coordinate (1-5) at (1-1.73205081,0+3);
\coordinate (2-5) at (0.5-1.73205081, 0.86602540378+3);
\coordinate (3-5) at (-0.5-1.73205081, 0.86602540378+3);
\coordinate (4-5) at (-1-1.73205081,0+3);
\coordinate (5-5) at (-0.5-1.73205081, -0.86602540378+3);
\coordinate (6-5) at (0.5-1.73205081, -0.86602540378+3);

\draw (Center-1) -- (1-1);
\draw (Center-1) -- (2-1);
\draw (Center-1) -- (3-1);
\draw (Center-1) -- (4-1);
\draw (Center-1) -- (5-1);
\draw (Center-1) -- (6-1);

\draw (Center-2) -- (1-2);
\draw (Center-2) -- (2-2);
\draw (Center-2) -- (3-2);
\draw (Center-2) -- (4-2);
\draw (Center-2) -- (5-2);
\draw (Center-2) -- (6-2);

\draw (Center-3) -- (1-3);
\draw (Center-3) -- (2-3);
\draw (Center-3) -- (3-3);
\draw (Center-3) -- (4-3);
\draw (Center-3) -- (5-3);
\draw (Center-3) -- (6-3);

\draw (Center-4) -- (1-4);
\draw (Center-4) -- (2-4);
\draw (Center-4) -- (3-4);
\draw (Center-4) -- (4-4);
\draw (Center-4) -- (5-4);
\draw (Center-4) -- (6-4);

\draw (Center-5) -- (1-5);
\draw (Center-5) -- (2-5);
\draw (Center-5) -- (3-5);
\draw (Center-5) -- (4-5);
\draw (Center-5) -- (5-5);
\draw (Center-5) -- (6-5);

\node at (Center-2) [above = 5mm] {$d$};
\node at (Center-2) [below = 5mm] {$g$};
\node at (Center-2) [above = 13 mm] {$e$};
\node at (5-2) [above = 2mm] {$f$};
\node at (6-2) [above = 2mm] {$h$};

\node at (Center-1) [shift={(0.4,0.25)}] {$\beta$};
\node at (Center-2) [shift={(0.4,0.2)}] {$\alpha$};
\node at (Center-3) [shift={(0.4,0.2)}] {$\gamma$};
\node at (Center-4) [shift={(-0.4,0.2)}] {$\eps$};
\node at (Center-5) [shift={(0.4,0.2)}] {$\delta$};

\draw (3-1) to [out=120, in=180 + 120] (6-2);
\draw (2-3) to [out=60, in=180+60] (5-2);
\draw (2-2) to [out=60, in=120](3-2);   
\draw (4-2) to [out=180, in=-60](6-5);  
\draw (1-2) to [out=0, in=240](5-4);  
\draw[dashed,->] (Center-5) to [out=-30, in=30+180] (Center-4) to (0.866 + 1.73205081, 0.5 + 3);

\end{tikzpicture}
        }
        \caption{\label{nonebigon} Face $g$ is not a bigon, and face $e$ has more than 3 edges}
\end{figure}

All cases work, meaning that a knot $K$ that is not trivial, trefoil, or figure-eight with a minimal 3-crossing diagram containing one monogon satisfies $c_9(K)\leq c_3(K)-2$.

\end{proof}

\begin{remark}
Using the 3-crossing knot tabulation done in \cite{tabulation}, we find that Lemma~\ref{twomonogons} and Lemma~\ref{onemonogon} account for all prime knots $K$ with $c_3(K)=3$. The only prime knots $K$ with $c_3(K)=4$ that Lemma~\ref{twomonogons}~and~\ref{onemonogon} do not account for are the $10_{140}$, $11n_{139}$, and $12n_{462}$ knots. Lemma~\ref{twomonogons}~and~\ref{onemonogon} account for all prime knots $K$ with $c_3(K)=5$.
\end{remark}

\begin{lemma}\label{zeromonogon}
Let $K$ be a knot that is not the trivial, trefoil, or figure-eight knot. If a minimal 3-crossing diagram of $K$ contains zero monogons, then $c_9(K)\leq c_3(K)-2$.
\end{lemma}

\begin{proof}
Consider a 3-crossing diagram of $K$ with $c_3(K)$ crossings. Let $f_i$ represent the number of faces with $i$ edges in a 3-crossing diagram. This includes the outer region. Adams, Hoste, \& Palmer prove in \cite{tripmoves} that $2f_1+f_2=6+f_4+2f_5+3f_6+4f_7+...$ using the Euler Characteristic. Since there are no monogons, we have
\begin{equation}\label{eulerchar}
f_2=6+f_4+2f_5+3f_6+4f_7+\dots.
\end{equation}
We will use proof by contradiction by assuming $c_9(K)>c_3(K)-2$. We will then derive a contradiction with \eqref{eulerchar}, proving the claim.\\

\begin{figure}[!htb]
        \center{\center{\begin{tabular}{lll}
    \raisebox{-.5\height}{
\begin{tikzpicture}[scale=0.7]

\coordinate (Center-1) at (0, 0);
\coordinate (1-1) at (0.86602540378,0.5);
\coordinate (2-1) at (0, 1);
\coordinate (3-1) at (-0.86602540378, 0.5);
\coordinate (4-1) at (-0.86602540378, -0.5);
\coordinate (5-1) at (0,-1);
\coordinate (6-1) at (0.86602540378, -0.5);

\coordinate (Center-2) at (0, 0+3);
\coordinate (1-2) at (0.86602540378,0.5+3);
\coordinate (2-2) at (0, 1+3);
\coordinate (3-2) at (-0.86602540378, 0.5+3);
\coordinate (4-2) at (-0.86602540378, -0.5+3);
\coordinate (5-2) at (0,-1+3);
\coordinate (6-2) at (0.86602540378, -0.5+3);

\draw (Center-1) -- (1-1);
\draw (Center-1) -- (2-1);
\draw (Center-1) -- (3-1);
\draw (Center-1) -- (4-1);
\draw (Center-1) -- (5-1);
\draw (Center-1) -- (6-1);

\draw (Center-2) -- (1-2);
\draw (Center-2) -- (2-2);
\draw (Center-2) -- (3-2);
\draw (Center-2) -- (4-2);
\draw (Center-2) -- (5-2);
\draw (Center-2) -- (6-2);

\draw (2-1) -- (5-2);
\draw (1-1) to [out=30, in=-30] (6-2);
\draw (3-1) to [out=150, in=210] (4-2);

\end{tikzpicture}

} & $\to$  & \raisebox{-.5\height}{

\begin{tikzpicture}[scale=0.7]

\coordinate (Center-1) at (0, 0);
\coordinate (1-1) at (30:0.6);
\coordinate (2-1) at (0, 1);
\coordinate (3-1) at (-0.86602540378, 0.5);
\coordinate (4-1) at (-0.86602540378, -0.5);
\coordinate (5-1) at (0,-1);
\coordinate (6-1) at (0.86602540378, -0.5);

\coordinate (Center-2) at (0, 0+3);
\coordinate (1-2) at (0.86602540378,0.5+3);
\coordinate (2-2) at (0, 1+3);
\coordinate (3-2) at (-0.86602540378, 0.5+3);
\coordinate (4-2) at (-0.86602540378, -0.5+3);
\coordinate (5-2) at (0,-1+3);
\coordinate (6-2) at (0.86602540378, -0.5+3);

\draw (Center-1) -- (1-1);
\draw (Center-1) -- (2-1);
\draw (Center-1) -- (3-1);
\draw (Center-1) -- (4-1);
\draw (Center-1) -- (5-1);
\draw (Center-1) -- (6-1);

\draw (2-1) to [out=90, in=180-50] (45:2) to [out=-45, in=0] (1, 0) to (Center-1) to [out=180, in=270](-1.5, 1.5) to [out=90, in = 270] (2-2);
\draw (1-1) to [out=30, in=180-22.5] (22.5:1) to [out=-22.5, in=12.5] (Center-1) to [out=180+12.5, in=270](-1.8, 1.8) to [out=90, in = 180 + 180 -45] (3-2);
\draw (3-1) to [out=150, in=270] (-1.2, 1.2) to [out=90, in = 180 + 45] (1-2);

\end{tikzpicture}}\\
\end{tabular}}}
        \caption{\label{twobigons} Two adjacent bigons reduce to a 5 crossing with an adjoined bigon}
\end{figure}
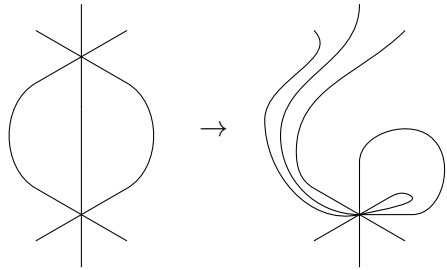

\begin{figure}[!htb]
\center{\begin{tabular}{lll}
        \raisebox{-.5\height}{
        \begin{tikzpicture}[scale=0.5]
        
        \foreach \vpos in {-3, 3} {
        \foreach \hpos in {-2, 2} {
                \coordinate (Center-\hpos-\vpos) at (\hpos, \vpos);
                \coordinate (1-\hpos-\vpos) at (\hpos + 1, \vpos + 0);
                \coordinate (2-\hpos-\vpos) at (\hpos + 0.5, \vpos + 0.86602540378);
                \coordinate (3-\hpos-\vpos) at (\hpos + -0.5, \vpos + 0.86602540378);
                \coordinate (4-\hpos-\vpos) at (\hpos + -1, \vpos + 0);
                \coordinate (5-\hpos-\vpos) at (\hpos + -0.5, \vpos + -0.86602540378);
                \coordinate (6-\hpos-\vpos) at (\hpos + 0.5, \vpos + -0.86602540378);
        
                \draw (Center-\hpos-\vpos) -- (1-\hpos-\vpos);
                \draw (Center-\hpos-\vpos) -- (2-\hpos-\vpos);
                \draw (Center-\hpos-\vpos) -- (3-\hpos-\vpos);
                \draw (Center-\hpos-\vpos) -- (4-\hpos-\vpos);
                \draw (Center-\hpos-\vpos) -- (5-\hpos-\vpos);
                \draw (Center-\hpos-\vpos) -- (6-\hpos-\vpos);
        
        }
        }
        
        \foreach \vpos in {0} {
        \foreach \hpos in {-1, 1} {
                \coordinate (Center-\hpos-\vpos) at (\hpos*3.7320508076, \vpos);
                \coordinate (1-\hpos-\vpos) at (\hpos*3.7320508076 + 1, \vpos + 0);
                \coordinate (2-\hpos-\vpos) at (\hpos*3.7320508076 + 0.5, \vpos + 0.86602540378);
                \coordinate (3-\hpos-\vpos) at (\hpos*3.7320508076 + -0.5, \vpos + 0.86602540378);
                \coordinate (4-\hpos-\vpos) at (\hpos*3.7320508076 + -1, \vpos + 0);
                \coordinate (5-\hpos-\vpos) at (\hpos*3.7320508076 + -0.5, \vpos + -0.86602540378);
                \coordinate (6-\hpos-\vpos) at (\hpos*3.7320508076 + 0.5, \vpos + -0.86602540378);
        
                \draw (Center-\hpos-\vpos) -- (1-\hpos-\vpos);
                \draw (Center-\hpos-\vpos) -- (2-\hpos-\vpos);
                \draw (Center-\hpos-\vpos) -- (3-\hpos-\vpos);
                \draw (Center-\hpos-\vpos) -- (4-\hpos-\vpos);
                \draw (Center-\hpos-\vpos) -- (5-\hpos-\vpos);
                \draw (Center-\hpos-\vpos) -- (6-\hpos-\vpos);
        
        }
        }
        
        \draw (1--2--3) to (4-2--3);
        \draw (1--2-3) to (4-2-3);
        \draw (3--2--3) to (6--1-0);
        \draw (2-2--3) to (5-1-0);
        \draw (6-2-3) to (3-1-0);
        \draw (5--2-3) to (2--1-0);
        
        \draw (2--2--3) to [out=60, in=120](3-2--3);
        \draw (5-2-3) to [out=240, in=180](4-1-0);
        \draw (6--2-3) to [out=300, in=0](1--1-0);
        
        \end{tikzpicture}
        
        } & $\to$  & \raisebox{-.5\height}{
        
        \begin{tikzpicture}[scale=0.5]
        
        \foreach \vpos in {-3, 3} {
        \foreach \hpos in {-2, 2} {
                \coordinate (Center-\hpos-\vpos) at (\hpos, \vpos);
                \coordinate (1-\hpos-\vpos) at (\hpos + 1, \vpos + 0);
                \coordinate (2-\hpos-\vpos) at (\hpos + 0.5, \vpos + 0.86602540378);
                \coordinate (3-\hpos-\vpos) at (\hpos + -0.5, \vpos + 0.86602540378);
                \coordinate (4-\hpos-\vpos) at (\hpos + -1, \vpos + 0);
                \coordinate (5-\hpos-\vpos) at (\hpos + -0.5, \vpos + -0.86602540378);
                \coordinate (6-\hpos-\vpos) at (\hpos + 0.5, \vpos + -0.86602540378);

        }
        }
        
        \draw (Center--2-3) -- (1--2-3);
        \draw (Center--2-3) -- (2--2-3);
        \draw (Center--2-3) -- (3--2-3);
        \draw (Center--2-3) -- (4--2-3);
        \draw (Center--2-3) -- (5--2-3);
        \draw (Center--2-3) -- (6--2-3);
        
        \draw (Center--2--3) -- (1--2--3);
        \draw (Center--2--3) -- (2--2--3);
        \draw (Center--2--3) -- (3--2--3);
        \draw (Center--2--3) -- (4--2--3);
        \draw (Center--2--3) -- (5--2--3);
        \draw (Center--2--3) -- (6--2--3);
        
        \draw (Center-2--3) -- (1-2--3);
        \draw (Center-2--3) -- (2-2--3);
        \draw (Center-2--3) -- (3-2--3);
        \draw (Center-2--3) -- (4-2--3);
        \draw (Center-2--3) -- (5-2--3);
        \draw (Center-2--3) -- (6-2--3);
        
        \foreach \vpos in {0} {
        \foreach \hpos in {-1, 1} {
                \coordinate (Center-\hpos-\vpos) at (\hpos*3.7320508076, \vpos);
                \coordinate (1-\hpos-\vpos) at (\hpos*3.7320508076 + 1, \vpos + 0);
                \coordinate (2-\hpos-\vpos) at (\hpos*3.7320508076 + 0.5, \vpos + 0.86602540378);
                \coordinate (3-\hpos-\vpos) at (\hpos*3.7320508076 + -0.5, \vpos + 0.86602540378);
                \coordinate (4-\hpos-\vpos) at (\hpos*3.7320508076 + -1, \vpos + 0);
                \coordinate (5-\hpos-\vpos) at (\hpos*3.7320508076 + -0.5, \vpos + -0.86602540378);
                \coordinate (6-\hpos-\vpos) at (\hpos*3.7320508076 + 0.5, \vpos + -0.86602540378);
        
                \draw (Center-\hpos-\vpos) -- (1-\hpos-\vpos);
                \draw (Center-\hpos-\vpos) -- (2-\hpos-\vpos);
                \draw (Center-\hpos-\vpos) -- (3-\hpos-\vpos);
                \draw (Center-\hpos-\vpos) -- (4-\hpos-\vpos);
                \draw (Center-\hpos-\vpos) -- (5-\hpos-\vpos);
                \draw (Center-\hpos-\vpos) -- (6-\hpos-\vpos);
        
        }
        }
        
        \draw (1--2--3) to (4-2--3);
        \draw (1--2-3) to [out=0, in=190](1-2-3);
        \draw (3--2--3) to (6--1-0);
        \draw (2-2--3) to (5-1-0);
        \draw (5--2-3) to (2--1-0);
        
        \draw (2--2--3) to [out=60, in=120](3-2--3);
        \draw (6--2-3) to [out=300, in=0](1--1-0);
        
        \draw (4-1-0) to [out=180, in=180](3.7320508076-0.2, 1.8) to [out=0, in=70](Center-1-0) to [out=250, in=40](Center-2--3) to [out=200, in=-20](Center--2--3) to [out=160, in=-60-40](Center--1-0) to [out=20+60, in=-120-20](Center--2-3) to [out=40, in=-60-40](3-2-3);
        \draw (3-1-0) to [out=120, in=180](3.7320508076-0.4, 1.5) to [out=0, in=90](Center-1-0) to [out=270, in=20](Center-2--3) to [out=220, in=-40](Center--2--3) to [out=140, in=-60-20](Center--1-0) to [out=40+60, in=-120-40](Center--2-3) to [out=20, in=-60-20](2-2-3);
        
        \end{tikzpicture}}
        \end{tabular}}
        \caption{\label{evenf} An $n$-gon with even $n$}
\end{figure}
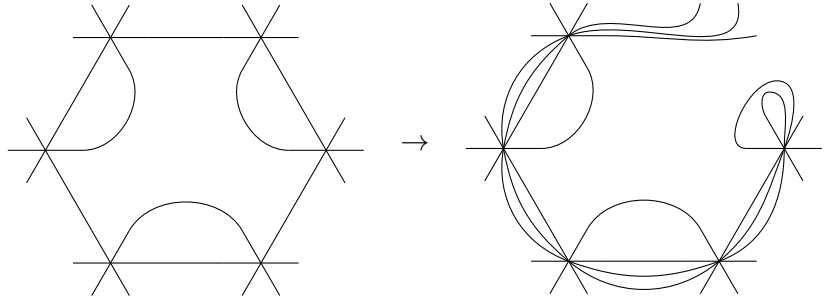

\begin{figure}[!htb]
        \center{\begin{tabular}{lll}
                \raisebox{-.5\height}{
            \begin{tikzpicture}[scale=0.7]
            
                \foreach \vpos in {-1} {
                    \foreach \hpos in {-1, 1} {
                        \coordinate (Center-\hpos-\vpos) at (3*0.809*\hpos, 3*0.309*\vpos);
                        \coordinate (1-\hpos-\vpos) at (3*0.809*\hpos + 1, 3*0.309*\vpos + 0);
                        \coordinate (2-\hpos-\vpos) at (3*0.809*\hpos + 0.5, 3*0.309*\vpos + 0.86602540378);
                        \coordinate (3-\hpos-\vpos) at (3*0.809*\hpos + -0.5, 3*0.309*\vpos + 0.86602540378);
                        \coordinate (4-\hpos-\vpos) at (3*0.809*\hpos + -1, 3*0.309*\vpos + 0);
                        \coordinate (5-\hpos-\vpos) at (3*0.809*\hpos + -0.5, 3*0.309*\vpos + -0.86602540378);
                        \coordinate (6-\hpos-\vpos) at (3*0.809*\hpos + 0.5, 3*0.309*\vpos + -0.86602540378);
                
                        \draw (Center-\hpos-\vpos) -- (1-\hpos-\vpos);
                        \draw (Center-\hpos-\vpos) -- (2-\hpos-\vpos);
                        \draw (Center-\hpos-\vpos) -- (3-\hpos-\vpos);
                        \draw (Center-\hpos-\vpos) -- (4-\hpos-\vpos);
                        \draw (Center-\hpos-\vpos) -- (5-\hpos-\vpos);
                        \draw (Center-\hpos-\vpos) -- (6-\hpos-\vpos);
                
                    }
                }
            
            \foreach \vpos in {1} {
                \foreach \hpos in {-1, 1} {
                    \coordinate (Center-\hpos-\vpos) at (\hpos*0.951*3, 3*0.5877*\vpos);
                    \coordinate (1-\hpos-\vpos) at (\hpos*0.951*3 + 1, 3*0.5877*\vpos + 0);
                    \coordinate (2-\hpos-\vpos) at (\hpos*0.951*3 + 0.5, 3*0.5877*\vpos + 0.86602540378);
                    \coordinate (3-\hpos-\vpos) at (\hpos*0.951*3 + -0.5, 3*0.5877*\vpos + 0.86602540378);
                    \coordinate (4-\hpos-\vpos) at (\hpos*0.951*3 + -1, 3*0.5877*\vpos + 0);
                    \coordinate (5-\hpos-\vpos) at (\hpos*0.951*3 + -0.5, 3*0.5877*\vpos + -0.86602540378);
                    \coordinate (6-\hpos-\vpos) at (\hpos*0.951*3 + 0.5, 3*0.5877*\vpos + -0.86602540378);
            
                    \draw (Center-\hpos-\vpos) -- (1-\hpos-\vpos);
                    \draw (Center-\hpos-\vpos) -- (2-\hpos-\vpos);
                    \draw (Center-\hpos-\vpos) -- (3-\hpos-\vpos);
                    \draw (Center-\hpos-\vpos) -- (4-\hpos-\vpos);
                    \draw (Center-\hpos-\vpos) -- (5-\hpos-\vpos);
                    \draw (Center-\hpos-\vpos) -- (6-\hpos-\vpos);
            
                }
            }
            
            \foreach \vpos in {1} {
                \foreach \hpos in {0} {
                    \coordinate (Center-\hpos-\vpos) at (\hpos, 3*\vpos);
                    \coordinate (1-\hpos-\vpos) at (\hpos + 1, 3*\vpos + 0);
                    \coordinate (2-\hpos-\vpos) at (\hpos + 0.5, 3*\vpos + 0.86602540378);
                    \coordinate (3-\hpos-\vpos) at (\hpos + -0.5, 3*\vpos + 0.86602540378);
                    \coordinate (4-\hpos-\vpos) at (\hpos + -1, 3*\vpos + 0);
                    \coordinate (5-\hpos-\vpos) at (\hpos + -0.5, 3*\vpos + -0.86602540378);
                    \coordinate (6-\hpos-\vpos) at (\hpos + 0.5, 3*\vpos + -0.86602540378);
            
                    \draw (Center-\hpos-\vpos) -- (1-\hpos-\vpos);
                    \draw (Center-\hpos-\vpos) -- (2-\hpos-\vpos);
                    \draw (Center-\hpos-\vpos) -- (3-\hpos-\vpos);
                    \draw (Center-\hpos-\vpos) -- (4-\hpos-\vpos);
                    \draw (Center-\hpos-\vpos) -- (5-\hpos-\vpos);
                    \draw (Center-\hpos-\vpos) -- (6-\hpos-\vpos);
            
                }
            }
            
            \draw (1--1--1) to (4-1--1);
            \draw (6--1--1) to [out=300, in=240] (5-1--1);
            \draw (2--1--1) to [out=60, in=300] (6--1-1);
            \draw (3-1--1) to [out=120, in=240] (5-1-1);
            
            \draw (4-1-1) to [out=180, in=300] (6-0-1);
            \draw (1--1-1) to [out=0, in=240] (5-0-1);
            \draw (3-1-1) to [out=120, in=0] (1-0-1);
            \draw (2--1-1) to [out=60, in=180] (4-0-1);
            
            \end{tikzpicture}
            
            } & $\to$  & \raisebox{-.5\height}{
            
            \begin{tikzpicture}[scale=0.7]
            
                \foreach \vpos in {-1} {
                    \foreach \hpos in {-1} {
                        \coordinate (Center-\hpos-\vpos) at (3*0.809*\hpos, 3*0.309*\vpos);
                        \coordinate (1-\hpos-\vpos) at (3*0.809*\hpos + 1, 3*0.309*\vpos + 0);
                        \coordinate (2-\hpos-\vpos) at (3*0.809*\hpos + 0.5, 3*0.309*\vpos + 0.86602540378);
                        \coordinate (3-\hpos-\vpos) at (3*0.809*\hpos + -0.5, 3*0.309*\vpos + 0.86602540378);
                        \coordinate (4-\hpos-\vpos) at (3*0.809*\hpos + -1, 3*0.309*\vpos + 0);
                        \coordinate (5-\hpos-\vpos) at (3*0.809*\hpos + -0.5, 3*0.309*\vpos + -0.86602540378);
                        \coordinate (6-\hpos-\vpos) at (3*0.809*\hpos + 0.5, 3*0.309*\vpos + -0.86602540378);
                
                        \draw (Center-\hpos-\vpos) -- (1-\hpos-\vpos);
                        \draw (Center-\hpos-\vpos) -- (2-\hpos-\vpos);
                        \draw (Center-\hpos-\vpos) -- (3-\hpos-\vpos);
                        \draw (Center-\hpos-\vpos) -- (4-\hpos-\vpos);
                        \draw (Center-\hpos-\vpos) -- (5-\hpos-\vpos);
                
                    }
                }
                \foreach \vpos in {-1} {
                    \foreach \hpos in {1} {
                        \coordinate (Center-\hpos-\vpos) at (3*0.809*\hpos, 3*0.309*\vpos);
                        \coordinate (1-\hpos-\vpos) at (3*0.809*\hpos + 1, 3*0.309*\vpos + 0);
                        \coordinate (2-\hpos-\vpos) at (3*0.809*\hpos + 0.5, 3*0.309*\vpos + 0.86602540378);
                        \coordinate (3-\hpos-\vpos) at (3*0.809*\hpos + -0.5, 3*0.309*\vpos + 0.86602540378);
                        \coordinate (4-\hpos-\vpos) at (3*0.809*\hpos + -1, 3*0.309*\vpos + 0);
                        \coordinate (5-\hpos-\vpos) at (3*0.809*\hpos + -0.5, 3*0.309*\vpos + -0.86602540378);
                        \coordinate (6-\hpos-\vpos) at (3*0.809*\hpos + 0.5, 3*0.309*\vpos + -0.86602540378);
                
                        \draw (Center-\hpos-\vpos) -- (1-\hpos-\vpos);
                        \draw (Center-\hpos-\vpos) -- (2-\hpos-\vpos);
                        \draw (Center-\hpos-\vpos) -- (3-\hpos-\vpos);
                        \draw (Center-\hpos-\vpos) -- (4-\hpos-\vpos);
                        \draw (Center-\hpos-\vpos) -- (6-\hpos-\vpos);
                
                    }
                }
            
            \foreach \vpos in {1} {
                \foreach \hpos in {-1} {
                    \coordinate (Center-\hpos-\vpos) at (\hpos*0.951*3, 3*0.5877*\vpos);
                    \coordinate (1-\hpos-\vpos) at (\hpos*0.951*3 + 1, 3*0.5877*\vpos + 0);
                    \coordinate (2-\hpos-\vpos) at (\hpos*0.951*3 + 0.5, 3*0.5877*\vpos + 0.86602540378);
                    \coordinate (3-\hpos-\vpos) at (\hpos*0.951*3 + -0.5, 3*0.5877*\vpos + 0.86602540378);
                    \coordinate (4-\hpos-\vpos) at (\hpos*0.951*3 + -1, 3*0.5877*\vpos + 0);
                    \coordinate (5-\hpos-\vpos) at (\hpos*0.951*3 + -0.5, 3*0.5877*\vpos + -0.86602540378);
                    \coordinate (6-\hpos-\vpos) at (\hpos*0.951*3 + 0.5, 3*0.5877*\vpos + -0.86602540378);
            
                    \draw (Center-\hpos-\vpos) -- (1-\hpos-\vpos);
                    \draw (Center-\hpos-\vpos) -- (3-\hpos-\vpos);
                    \draw (Center-\hpos-\vpos) -- (4-\hpos-\vpos);
                    \draw (Center-\hpos-\vpos) -- (5-\hpos-\vpos);
                    \draw (Center-\hpos-\vpos) -- (6-\hpos-\vpos);
            
                }
            }
            
            \foreach \vpos in {1} {
                \foreach \hpos in {1} {
                    \coordinate (Center-\hpos-\vpos) at (\hpos*0.951*3, 3*0.5877*\vpos);
                    \coordinate (1-\hpos-\vpos) at (\hpos*0.951*3 + 1, 3*0.5877*\vpos + 0);
                    \coordinate (2-\hpos-\vpos) at (\hpos*0.951*3 + 0.5, 3*0.5877*\vpos + 0.86602540378);
                    \coordinate (3-\hpos-\vpos) at (\hpos*0.951*3 + -0.5, 3*0.5877*\vpos + 0.86602540378);
                    \coordinate (4-\hpos-\vpos) at (\hpos*0.951*3 + -1, 3*0.5877*\vpos + 0);
                    \coordinate (5-\hpos-\vpos) at (\hpos*0.951*3 + -0.5, 3*0.5877*\vpos + -0.86602540378);
                    \coordinate (6-\hpos-\vpos) at (\hpos*0.951*3 + 0.5, 3*0.5877*\vpos + -0.86602540378);
            
                    \draw (Center-\hpos-\vpos) -- (2-\hpos-\vpos);
                    \draw (Center-\hpos-\vpos) -- (5-\hpos-\vpos);
            
                }
            }
            
            \foreach \vpos in {1} {
                \foreach \hpos in {0} {
                    \coordinate (Center-\hpos-\vpos) at (\hpos, 3*\vpos);
                    \coordinate (1-\hpos-\vpos) at (\hpos + 1, 3*\vpos + 0);
                    \coordinate (2-\hpos-\vpos) at (\hpos + 0.5, 3*\vpos + 0.86602540378);
                    \coordinate (3-\hpos-\vpos) at (\hpos + -0.5, 3*\vpos + 0.86602540378);
                    \coordinate (4-\hpos-\vpos) at (\hpos + -1, 3*\vpos + 0);
                    \coordinate (5-\hpos-\vpos) at (\hpos + -0.5, 3*\vpos + -0.86602540378);
                    \coordinate (6-\hpos-\vpos) at (\hpos + 0.5, 3*\vpos + -0.86602540378);
            
                    \draw (Center-\hpos-\vpos) -- (1-\hpos-\vpos);
                    \draw (Center-\hpos-\vpos) -- (2-\hpos-\vpos);
                    \draw (Center-\hpos-\vpos) -- (3-\hpos-\vpos);
                    \draw (Center-\hpos-\vpos) -- (5-\hpos-\vpos);
                    \draw (Center-\hpos-\vpos) -- (6-\hpos-\vpos);
            
                }
            }
            
            \draw (1--1--1) to (4-1--1);
            \draw (Center--1--1) to [out=300, in=240] (Center-1--1);
            \draw (2--1--1) to [out=60, in=300] (6--1-1);
            \draw (3-1--1) to [out=120, in=240] (5-1-1);
            
            \draw (4-1-1) to [out=180, in=300] (6-0-1);
            \draw (1--1-1) to [out=0, in=240] (5-0-1);
            \draw (3-1-1) to [out=120, in=0] (1-0-1);
            \draw (Center--1-1) to [out=40, in=180] (Center-0-1);
            
            \draw (4-1-1) to [out=0, in=-40](Center-0-1) to [out=180-40, in=60+40](Center--1-1) to [out=-60-20, in=60+20](Center--1--1) to [out=-60-40, in=180] (0, -2.5) to [out=0, in=-60-20](Center-1--1) to [out=60+40, in=180] (1-1-1);
            \draw (3-1-1) to [out=-60, in=-20](Center-0-1) to [out=180-20, in=60+20](Center--1-1) to [out=-60-40, in=60+40](Center--1--1) to [out=-60-20, in=-60-40](Center-1--1) to [out=60+20, in=120] (6-1-1);
            
            \end{tikzpicture}}
            \end{tabular}}
        \caption{\label{oddf} An $n$-gon with odd $n$}
\end{figure}
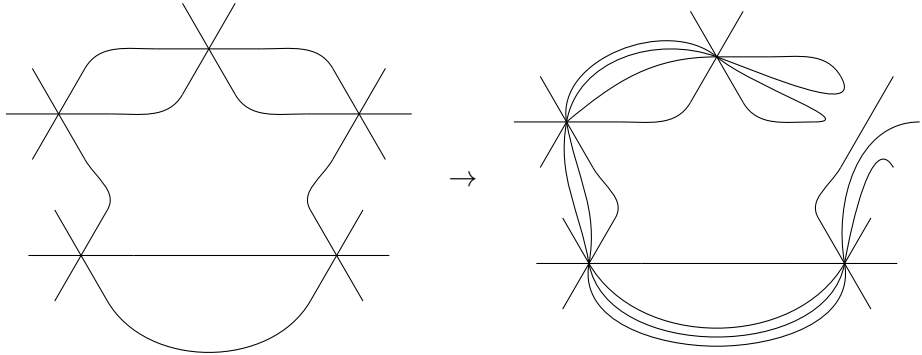

\begin{figure}[!htb]
        \center{\includegraphics[width=.5\textwidth]
        {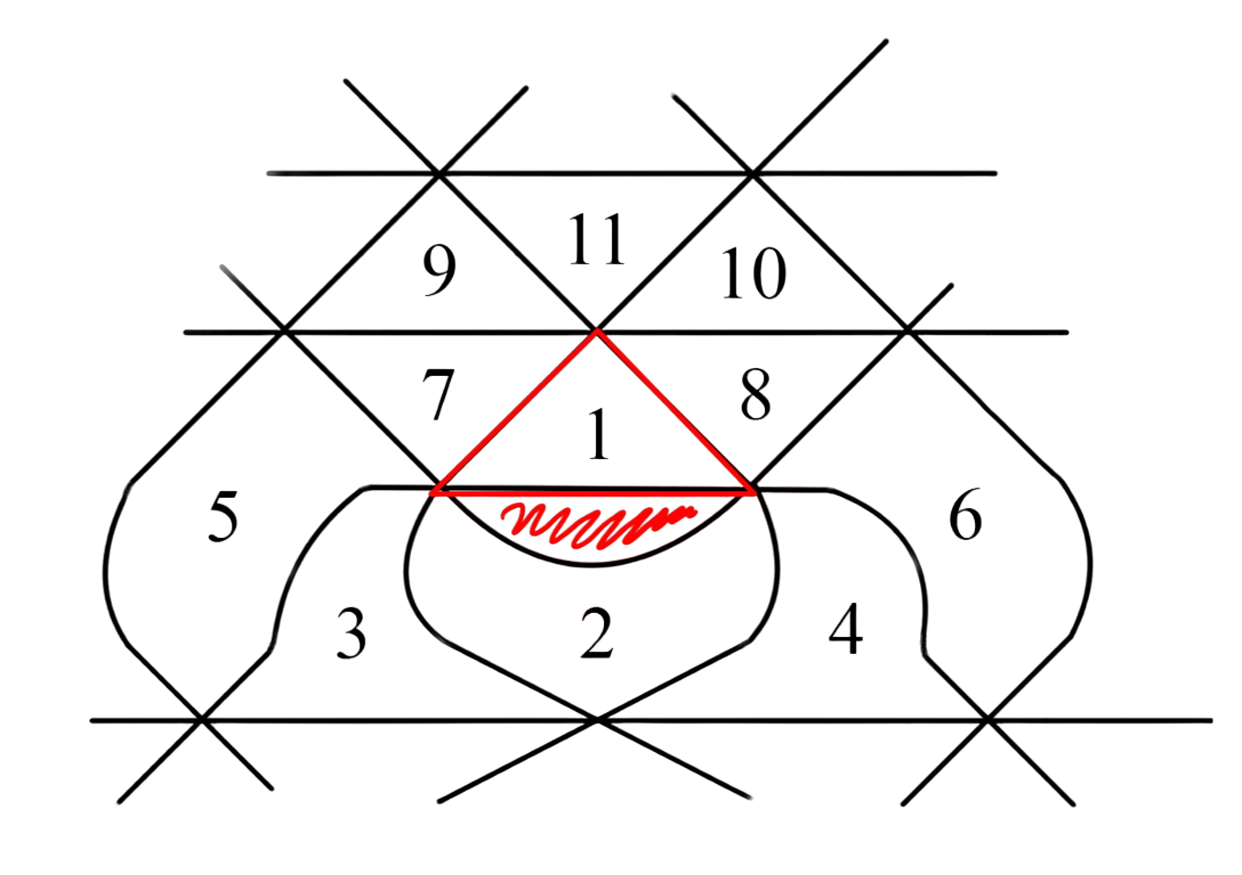}}
        \caption{\label{trianglecount} The maximum number of surrounding triangles}
\end{figure}

Note that two bigons can't be adjacent; otherwise, we can do the manipulations seen in Figure~\ref{twobigons}, meaning $c_9(K)\leq c_3(K)-2$. So, we can count the total number of bigons by counting the number of bigons that can border each $n$-gon. An $n$-gon must have fewer than $n-1$ adjacent bigons. If not, then there exists a path around the $n$-gon where all but one edge is adjacent to a bigon. This means $c_9(K)\leq c_3(K)-2$: If $n$ is even, $c_9(K)\leq c_3(K)-2$ by the construction seen in Figure~\ref{evenf}. And if $n$ is odd, $c_9(K)\leq c_3(K)-2$ by the construction seen in Figure~\ref{oddf}. So every $n$-gon borders at most $n-2$ bigons. In counting bigons this way, we have counted them all twice. So, we obtain the inequality: $$f_2\leq \frac{1}{2}f_3+f_4+\frac{3}{2}f_5+2f_6+\dots$$ To improve this bound and thus reach a contradiction with \eqref{eulerchar}, look towards each triangle in the knot projection. If two triangles share no vertices and are both adjacent to a bigon, we are done. Now, pick a triangle with an adjacent bigon—if no triangles have an adjacent bigon, $$f_2\leq f_4+\frac{3}{2}f_5+2f_6+\dots$$and we achieve a contradiction with \eqref{eulerchar} regardless. There are at most 10 other triangles that can share a crossing with our first triangle (see Figure~\ref{trianglecount}), for a total of 11 triangles. Each of these triangles can have at most one adjacent bigon. Our new bound becomes $$f_2\leq \frac{11+2f_4+3f_5+4f_6+\dots}{2}.$$So,
\begin{equation}\label{newineq}
f_2\leq \frac{11}{2}+f_4+\frac{3}{2}f_5+2f_6+\dots.
\end{equation}
This contradicts \eqref{eulerchar}, meaning the claim must hold.
\end{proof}

\vspace{0.5 cm}
Thus, it is clear to see from Lemma~\ref{twomonogons}, Lemma~\ref{onemonogon}, and Lemma~\ref{zeromonogon} that: 
\begin{theorem}\label{endresult}
Let $K$ be a knot that is not the trivial, trefoil, or figure-eight knot. Then $$c_9(K)\leq c_3(K)-2$$ 
\end{theorem}

This inequality is optimal. \cite{tabulation} shows that $c_3(5_1)=3$ and $c_3(6_2)=3$. Since $c_9(5_1)=1$ and $c_9(6_2)=1$, these knots realize the upper bound.

\section{A 13-crossing number inequality}\label{section13crossing}

Let $K$ be a knot that is not the trivial, trefoil, or figure-eight knot. It is already known that $c_5(K) \leq c_3(K)-1$ \cite{tripmoves}. If we assume that $c_9(K) \leq c_5(K) - 1$, it is clear that $c_9(K)\leq c_3(K)-2$. Theorem~\ref{endresult} proves that $c_9(K)\leq c_3(K)-2$. In this section, we will make progress towards proving inequalities concerning the 5-crossing number, such as $c_9(K)\leq c_5(K)-1$ and $c_{13}(K)\leq c_5(K)-1$.\\

By Corollary~\ref{accoccremove}, $c_9(K)<c_5(K)$ if a minimal 5-crossing diagram of $K$ contains two adjacent or opposite crossing connected faces. Theorem~\ref{endsection1proof} proves that there won't exist ACC or OCC faces in a 5-crossing diagram only if every crossing contains no monogons, one monogon, or two almost opposite monogons. We might expect that any 5-crossing projection with more than one crossing has two faces that are adjacent or opposite crossing connected; after all, this is the case for all 3-crossing diagrams. However, Figure~\ref{9counter} provides a counter-example to this conjecture. In this section, we will instead make progress on proving the following conjecture:

\begin{figure}[!htb]
        \center{
\begin{tikzpicture}[scale=0.85]

\foreach \position in {-2, 2} {
    \coordinate (Center-\position) at (0+\position, 0);

    \coordinate (1-\position) at (0.809+\position, 0.588);
    \coordinate (2-\position) at (0.309+\position, 0.951);
    \coordinate (3-\position) at (-0.309+\position, 0.951);
    \coordinate (4-\position) at (-0.809+\position,0.588);
    \coordinate (5-\position) at (-1+\position,0);
    \coordinate (6-\position) at (-0.809+\position,-.588);
    \coordinate (7-\position) at (-0.309+\position,-.951);
    \coordinate (8-\position) at (0.309+\position,-.951);
    \coordinate (9-\position) at (0.809+\position,-.588);
    \coordinate (10-\position) at (1+\position,0);

    \draw (Center-\position) -- (1-\position);
    \draw (Center-\position) -- (2-\position);
    \draw (Center-\position) -- (3-\position);
    \draw (Center-\position) -- (4-\position);
    \draw (Center-\position) -- (5-\position);
    \draw (Center-\position) -- (6-\position);
    \draw (Center-\position) -- (7-\position);
    \draw (Center-\position) -- (8-\position);
    \draw (Center-\position) -- (9-\position);
    \draw (Center-\position) -- (10-\position);
}
\foreach \position in {0}{
    \coordinate (Center-\position) at (0+\position, 0-3);

    \coordinate (1-\position) at (0.809+\position, 0.588-3);
    \coordinate (2-\position) at (0.309+\position, 0.951-3);
    \coordinate (3-\position) at (-0.309+\position, 0.951-3);
    \coordinate (4-\position) at (-0.809+\position,0.588-3);
    \coordinate (5-\position) at (-1+\position,0-3);
    \coordinate (6-\position) at (-0.809+\position,-.588-3);
    \coordinate (7-\position) at (-0.309+\position,-.951-3);
    \coordinate (8-\position) at (0.309+\position,-.951-3);
    \coordinate (9-\position) at (0.809+\position,-.588-3);
    \coordinate (10-\position) at (1+\position,0-3);
    
    \draw (Center-\position) -- (1-\position);
    \draw (Center-\position) -- (2-\position);
    \draw (Center-\position) -- (3-\position);
    \draw (Center-\position) -- (4-\position);
    \draw (Center-\position) -- (5-\position);
    \draw (Center-\position) -- (6-\position);
    \draw (Center-\position) -- (7-\position);
    \draw (Center-\position) -- (8-\position);
    \draw (Center-\position) -- (9-\position);
    \draw (Center-\position) -- (10-\position);
}

\draw (8-0) to [out=288, in=-36](9-0);
\draw (8--2) to [out=288, in=-36](9--2);
\draw (4--2) to [out=144, in=180](5--2);

\draw (5-0) to [out=180, in=-144](6--2);
\draw (4-0) to [out=144, in=-108](7--2);
\draw (3-0) to [out=108, in=-144](6-2);
\draw (2-0) to [out=72, in=-108](7-2);
\draw (1-0) to [out=36, in=-72](8-2);
\draw (10-0) to [out=0, in=-36](9-2);
\draw (7-0) to [out=-108, in=180+60] (4, -2) to [out=60, in=0](10-2);
\draw (6-0) to [out=-144, in=180] (0, -5) to [out=0, in=180+60] (4.5, -2.5) to [out=60, in=36](1-2);

\draw (10--2) to [out=0, in=180] (5-2);
\draw (1--2) to [out=36, in=180-36] (4-2);
\draw (2--2) to [out=36*2, in=180-36*2] (3-2);
\draw (3--2) to [out=36*3, in=180-36*3] (2-2);

\end{tikzpicture}
        }
        \caption{\label{9counter} Not all 5-crossing projections contain two adjacent or opposite crossing connected faces}
\end{figure}
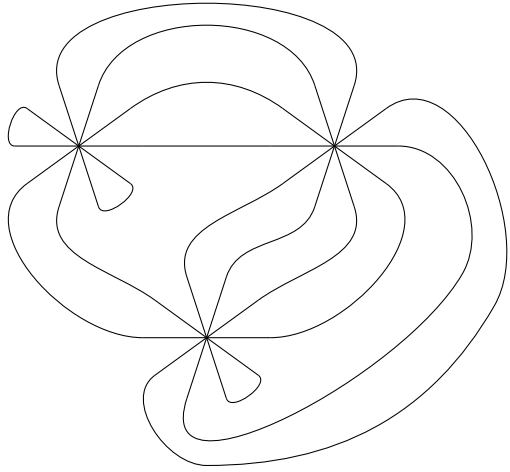

\begin{conjecture}\label{13ineqconj}
        For all knots $K$ such that $c_5(K)\geq2$, $c_{13}(K)<c_5(K)$.
\end{conjecture}

We focus our efforts on Conjecture~\ref{13ineqconj} because there is a clever method to transform the knot in Figure~\ref{9counter} and similar knots into 7-crossing projections that do contain two faces that are ACC or OCC. For knots on which this move is possible, $c_{13}(K)<c_5(K)$. It should also be noted that all knots on which $c_9(K)<c_5(K)$ satisfy $c_{13}(K)<c_5(K)$ as $c_{n+2}(K)\leq c_n(K)$.

We will label the faces of any crossing containing at least one monogon according to Figure~\ref{5onemonogonlabel}. If a crossing contains two almost opposite monogons, arbitrarily choose one monogon as the reference monogon.

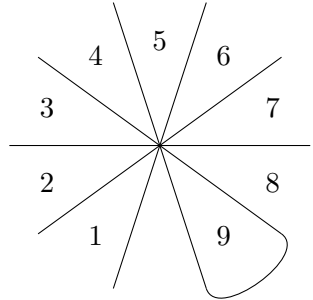
\begin{figure}[!htb]
        \center{
\begin{tikzpicture}

\coordinate (Center) at (0, 0);

\coordinate (1) at (0:2);
\coordinate (2) at (36:2);
\coordinate (3) at (72:2);
\coordinate (4) at (108:2);
\coordinate (5) at (144:2);
\coordinate (6) at (180:2);
\coordinate (7) at (216:2);
\coordinate (8) at (252:2);
\coordinate (9) at (288:2);
\coordinate (10) at (324:2);

\draw (Center) -- (1);
\draw (Center) -- (2);
\draw (Center) -- (3);
\draw (Center) -- (4);
\draw (Center) -- (5);
\draw (Center) -- (6);
\draw (Center) -- (7);
\draw (Center) -- (8);
\draw (Center) -- (9);
\draw (Center) -- (10);

\node at (Center) [shift={(-1.5,-0.5)}] {$2$};
\node at (Center) [shift={(-0.85,-1.2)}] {$1$};
\node at (Center) [shift={(0.85,-1.2)}] {$9$};
\node at (Center) [shift={(1.5,-0.5)}] {$8$};

\node at (Center) [shift={(-1.5,0.5)}] {$3$};
\node at (Center) [shift={(-0.85,1.2)}] {$4$};
\node at (Center) [shift={(0, 1.4)}] {$5$};
\node at (Center) [shift={(0.85,1.2)}] {$6$};
\node at (Center) [shift={(1.5,0.5)}] {$7$};

\draw (9) to [out=-72, in=-36] (10);

\end{tikzpicture}
        }
        \caption{\label{5onemonogonlabel} Labeling faces of 5-crossing with one monogon}
\end{figure}

\begin{lemma}\label{first13salvage}
Let $D$ be a minimal 5-crossing diagram of knot $K$. Suppose $D$ has crossing $\alpha$ containing at least one monogon. If there exists a crossing segment connecting face~$1$ and any face other than $4$ via $\alpha$, $c_{13}(K)<c_5(K)$.
\end{lemma}

\begin{proof}
A crossing segment that connects face 1 to face 2 or 8 forms an ACC, meaning we are done. In addition, a crossing segment that connects face 1 to face 6 forms an OCC, which also implies that $c_{13}(K)<c_5(K)$. Next, consider the case where face 1 is connected to face 3 by a crossing segment. Pulling two strands around $\alpha$ and the given crossing segment forms a 7-crossing where face 7 and 8 are ACC via some adjacent crossing $\beta$. This is visualized in Figure~\ref{onethreecrossingconnected}. One can then apply the move seen in Figure~\ref{increasebytwofirsttime} to all remaining 5-crossings. Then, by Corollary~\ref{accoccremove}, $c_{13}(K)<c_5(K)$ for this specific case. In a similar manner, the moves shown in Figure~\ref{357crossingconnected} can be done if there exists a crossing segment connecting face 1 to face 5 or 7. Crossing $\alpha$ is transformed into a 7-crossing such that face 7 and 8 are ACC via crossing $\beta$, all remaining crossings are transformed into 7-crossings, and Corollary~\ref{accoccremove} gives that $c_{13}(K)<c_5(K)$.
\end{proof}

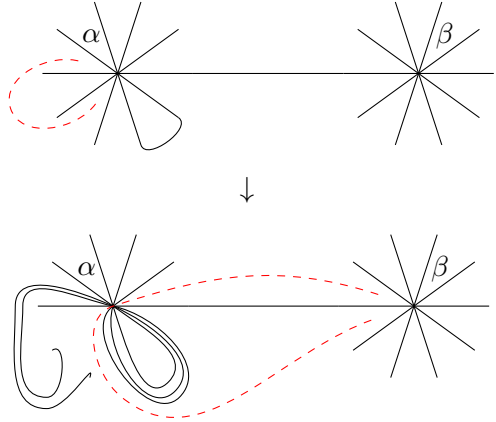
\begin{figure}[!htb]
        \center{
\begin{tikzpicture}
\coordinate (Center-1) at (0-2, 0);

\coordinate (1-1) at (0.809-2, 0.588);
\coordinate (2-1) at (0.309-2, 0.951);
\coordinate (3-1) at (-0.309-2, 0.951);
\coordinate (4-1) at (-0.809-2,0.588);
\coordinate (5-1) at (-1-2,0);
\coordinate (6-1) at (-0.809-2,-.588);
\coordinate (7-1) at (-0.309-2,-.951);
\coordinate (8-1) at (0.309-2,-.951);
\coordinate (9-1) at (0.809-2,-.588);
\coordinate (10-1) at (1-2,0);

\coordinate (Center-2) at (0+2, 0);

\coordinate (1-2) at (0.809+2, 0.588);
\coordinate (2-2) at (0.309+2, 0.951);
\coordinate (3-2) at (-0.309+2, 0.951);
\coordinate (4-2) at (-0.809+2,0.588);
\coordinate (5-2) at (-1+2,0);
\coordinate (6-2) at (-0.809+2,-.588);
\coordinate (7-2) at (-0.309+2,-.951);
\coordinate (8-2) at (0.309+2,-.951);
\coordinate (9-2) at (0.809+2,-.588);
\coordinate (10-2) at (1+2,0);

\draw (Center-1) -- (1-1);
\draw (Center-1) -- (2-1);
\draw (Center-1) -- (3-1);
\draw (Center-1) -- (4-1);
\draw (Center-1) -- (5-1);
\draw (Center-1) -- (6-1);
\draw (Center-1) -- (7-1);
\draw (Center-1) -- (8-1);
\draw (Center-1) -- (9-1);
\draw (Center-1) -- (10-1);

\draw (Center-2) -- (1-2);
\draw (Center-2) -- (2-2);
\draw (Center-2) -- (3-2);
\draw (Center-2) -- (4-2);
\draw (Center-2) -- (5-2);
\draw (Center-2) -- (6-2);
\draw (Center-2) -- (7-2);
\draw (Center-2) -- (8-2);
\draw (Center-2) -- (9-2);
\draw (Center-2) -- (10-2);


\node at (Center-1) [shift={(-0.35,0.5)}] {$\alpha$};
\node at (Center-2) [shift={(0.35,0.5)}] {$\beta$};

\draw (8-1) to [out=-72,in=-36] (9-1);
\draw (10-1) to (5-2);

\draw [dashed, red] (-0.2938-2, -0.4045) to [out=180+36+18, in = 180+108] (3*-0.475 - 2, 3*-0.1545) to [out = 108, in=180-18] (-0.475 - 2, 0.1545);


\end{tikzpicture}

$\downarrow$\\

\begin{tikzpicture}
    
\end{tikzpicture}

\begin{tikzpicture}
\coordinate (Center-1) at (0-2, 0);

\coordinate (1-1) at (0.809-2, 0.588);
\coordinate (2-1) at (0.309-2, 0.951);
\coordinate (3-1) at (-0.309-2, 0.951);
\coordinate (4-1) at (-0.809-2,0.588);
\coordinate (5-1) at (-1-2,0);
\coordinate (6-1) at (-0.809-2,-.588);
\coordinate (7-1) at (-0.309-2,-.951);
\coordinate (8-1) at (0.309-2,-.951);
\coordinate (9-1) at (0.809-2,-.588);
\coordinate (10-1) at (1-2,0);

\coordinate (Center-2) at (0+2, 0);

\coordinate (1-2) at (0.809+2, 0.588);
\coordinate (2-2) at (0.309+2, 0.951);
\coordinate (3-2) at (-0.309+2, 0.951);
\coordinate (4-2) at (-0.809+2,0.588);
\coordinate (5-2) at (-1+2,0);
\coordinate (6-2) at (-0.809+2,-.588);
\coordinate (7-2) at (-0.309+2,-.951);
\coordinate (8-2) at (0.309+2,-.951);
\coordinate (9-2) at (0.809+2,-.588);
\coordinate (10-2) at (1+2,0);

\draw (Center-1) -- (1-1);
\draw (Center-1) -- (2-1);
\draw (Center-1) -- (3-1);
\draw (Center-1) -- (4-1);
\draw (Center-1) -- (5-1);
\draw (Center-1) -- (10-1);

\draw (Center-2) -- (1-2);
\draw (Center-2) -- (2-2);
\draw (Center-2) -- (3-2);
\draw (Center-2) -- (4-2);
\draw (Center-2) -- (5-2);
\draw (Center-2) -- (6-2);
\draw (Center-2) -- (7-2);
\draw (Center-2) -- (8-2);
\draw (Center-2) -- (9-2);
\draw (Center-2) -- (10-2);

\node at (Center-1) [shift={(-0.35,0.5)}] {$\alpha$};
\node at (Center-2) [shift={(0.35,0.5)}] {$\beta$};

\draw (Center-1) to [out=-72, in=180+45] (2.5*0.2938-2, 2.5*-0.4045) to  [out=45, in=-36] (Center-1);
\draw (10-1) to (5-2);

\draw (6-1) to [out=36, in=0] (3*-0.2938-2, 3*-0.4045) to [out=180, in=270] (-3.2, 0) to [out=90, in= 180-12](Center-1) to [out=-12, in = 45] (3*0.2938-2, 3*-0.4045) to [out=180+45, in = 360-36*4](Center-1);

\draw (7-1) to [out=36*2, in=0] (3*-0.2938-2, 3*-0.4545) to [out=180, in=270] (-3.3, 0) to [out=90, in= 180-24](Center-1) to [out=-24, in = 45] (2.8*0.2938-2, 2.8*-0.4045) to [out=180+45, in = 360-36*3](Center-1);

\draw[dashed, red] (-0.475 + 2, 0.1545) to [out=180-18, in=18] (Center-1) to [out=180+18, in=140] (-1.9, -1.2) to [out=-40, in=180+18] (-0.475 + 2, -0.1545);

\end{tikzpicture}}
        \caption{\label{onethreecrossingconnected} Transforming a 5-crossing with a crossing segment between face 1 and face 3 into a 7-crossing with an ACC.}
\end{figure}

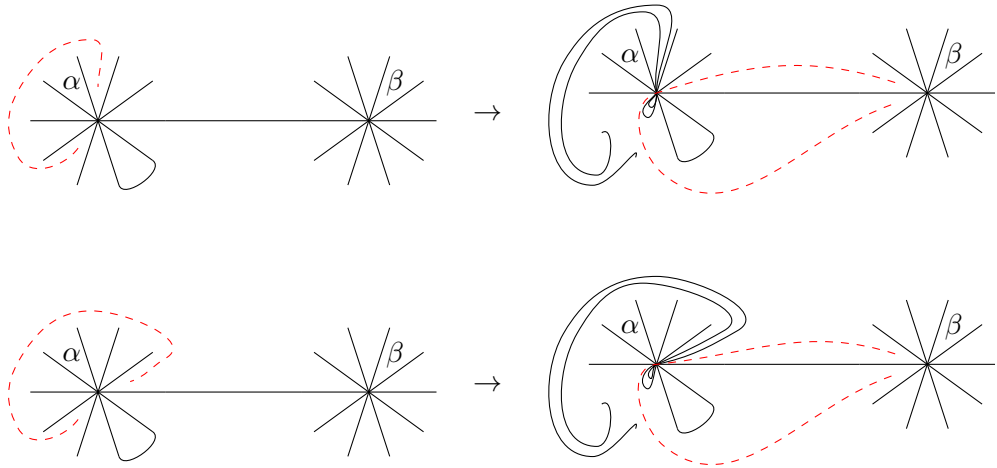
\begin{figure}[!htb]

        \center{\begin{tabular}{lll}
                \raisebox{-.5\height}{
            \begin{tikzpicture}[scale=0.9]
            
                \coordinate (Center-1) at (0-2, 0);
            
                \coordinate (1-1) at (0.809-2, 0.588);
                \coordinate (2-1) at (0.309-2, 0.951);
                \coordinate (3-1) at (-0.309-2, 0.951);
                \coordinate (4-1) at (-0.809-2,0.588);
                \coordinate (5-1) at (-1-2,0);
                \coordinate (6-1) at (-0.809-2,-.588);
                \coordinate (7-1) at (-0.309-2,-.951);
                \coordinate (8-1) at (0.309-2,-.951);
                \coordinate (9-1) at (0.809-2,-.588);
                \coordinate (10-1) at (1-2,0);
                
                \coordinate (Center-2) at (0+2, 0);
                
                \coordinate (1-2) at (0.809+2, 0.588);
                \coordinate (2-2) at (0.309+2, 0.951);
                \coordinate (3-2) at (-0.309+2, 0.951);
                \coordinate (4-2) at (-0.809+2,0.588);
                \coordinate (5-2) at (-1+2,0);
                \coordinate (6-2) at (-0.809+2,-.588);
                \coordinate (7-2) at (-0.309+2,-.951);
                \coordinate (8-2) at (0.309+2,-.951);
                \coordinate (9-2) at (0.809+2,-.588);
                \coordinate (10-2) at (1+2,0);

                \draw (Center-1) -- (1-1);
                \draw (Center-1) -- (2-1);
                \draw (Center-1) -- (3-1);
                \draw (Center-1) -- (4-1);
                \draw (Center-1) -- (5-1);
                \draw (Center-1) -- (6-1);
                \draw (Center-1) -- (7-1);
                \draw (Center-1) -- (8-1);
                \draw (Center-1) -- (9-1);
                \draw (Center-1) -- (10-1);

                \draw (Center-2) -- (1-2);
                \draw (Center-2) -- (2-2);
                \draw (Center-2) -- (3-2);
                \draw (Center-2) -- (4-2);
                \draw (Center-2) -- (5-2);
                \draw (Center-2) -- (6-2);
                \draw (Center-2) -- (7-2);
                \draw (Center-2) -- (8-2);
                \draw (Center-2) -- (9-2);
                \draw (Center-2) -- (10-2);
                
                
                \node at (Center-1) [shift={(-0.35,0.5)}] {$\alpha$};
                \node at (Center-2) [shift={(0.35,0.5)}] {$\beta$};

                \draw (8-1) to [out=-72,in=-36] (9-1);
                \draw (10-1) to (5-2);
                
                \draw [dashed, red] (-0.2938-2, -0.4045) to [out=180+36+18, in = 300] (-1.175 - 2, -0.545) to [out=120, in=60+180] (-1.175 - 2, 0.4545) to [out = 60, in=180] (-0.1 - 2, 1.2) to [out=0, in=90] (0 - 2, 0.5);
                
            
            \end{tikzpicture}
            
            } & $\to$  & \raisebox{-.5\height}{
            
            \begin{tikzpicture}[scale=0.9]
            
            \coordinate (Center-1) at (0-2, 0);
            
            \coordinate (1-1) at (0.809-2, 0.588);
            \coordinate (2-1) at (0.309-2, 0.951);
            \coordinate (3-1) at (-0.309-2, 0.951);
            \coordinate (4-1) at (-0.809-2,0.588);
            \coordinate (5-1) at (-1-2,0);
            \coordinate (6-1) at (-0.809-2,-.588);
            \coordinate (7-1) at (-0.309-2,-.951);
            \coordinate (8-1) at (0.309-2,-.951);
            \coordinate (9-1) at (0.809-2,-.588);
            \coordinate (10-1) at (1-2,0);
            
            \coordinate (Center-2) at (0+2, 0);
            
            \coordinate (1-2) at (0.809+2, 0.588);
            \coordinate (2-2) at (0.309+2, 0.951);
            \coordinate (3-2) at (-0.309+2, 0.951);
            \coordinate (4-2) at (-0.809+2,0.588);
            \coordinate (5-2) at (-1+2,0);
            \coordinate (6-2) at (-0.809+2,-.588);
            \coordinate (7-2) at (-0.309+2,-.951);
            \coordinate (8-2) at (0.309+2,-.951);
            \coordinate (9-2) at (0.809+2,-.588);
            \coordinate (10-2) at (1+2,0);

            \draw (Center-1) -- (1-1);
            \draw (Center-1) -- (2-1);
            \draw (Center-1) -- (3-1);
            \draw (Center-1) -- (4-1);
            \draw (Center-1) -- (5-1);
            \draw (Center-1) -- (8-1);
            \draw (Center-1) -- (9-1);
            \draw (Center-1) -- (10-1);

            \draw (Center-2) -- (1-2);
            \draw (Center-2) -- (2-2);
            \draw (Center-2) -- (3-2);
            \draw (Center-2) -- (4-2);
            \draw (Center-2) -- (5-2);
            \draw (Center-2) -- (6-2);
            \draw (Center-2) -- (7-2);
            \draw (Center-2) -- (8-2);
            \draw (Center-2) -- (9-2);
            \draw (Center-2) -- (10-2);
            
            \node at (Center-1) [shift={(-0.35,0.5)}] {$\alpha$};
            \node at (Center-2) [shift={(0.35,0.5)}] {$\beta$};

            \draw (8-1) to [out=-72, in=-36] (9-1);
            \draw (10-1) to (5-2);
            
            \draw (6-1) to [out=36, in=0] (3*-0.2938-2, 3*-0.4045) to [out=180, in=60+180] (-1.175 - 2, 0.4545) to [out = 60, in=180] (-0.2 - 2, 1.2) to [out=0, in=100] (Center-1) to [out=100+180, in=-45] (-2-.1, -0.2) to [out=-45+180, in=250] (Center-1);
            
            \draw (7-1) to [out=36*2, in=0] (3*-0.3138-2, 3*-0.4545) to [out=180, in=60+180] (-1.375 - 2, 0.4545) to [out = 60, in=180] (0 - 2, 1.3) to [out=0, in=90] (Center-1) to [out=90+180, in=-45] (-2-.1 * 1.7, -0.2 * 1.7) to [out=-45+180, in=230] (Center-1);
            
            

            \draw[dashed, red] (-0.475 + 2, 0.1545) to [out=180-18, in=18] (Center-1) to [out=180+18, in=140] (-1.9, -1.2) to [out=-40, in=180+18] (-0.475 + 2, -0.1545);
            
            \end{tikzpicture}}
            \end{tabular}
            
            \vspace{5mm}
            
            \begin{tabular}{lll}
                \raisebox{-.5\height}{
            \begin{tikzpicture}[scale=0.9]
            
                \coordinate (Center-1) at (0-2, 0);
            
                \coordinate (1-1) at (0.809-2, 0.588);
                \coordinate (2-1) at (0.309-2, 0.951);
                \coordinate (3-1) at (-0.309-2, 0.951);
                \coordinate (4-1) at (-0.809-2,0.588);
                \coordinate (5-1) at (-1-2,0);
                \coordinate (6-1) at (-0.809-2,-.588);
                \coordinate (7-1) at (-0.309-2,-.951);
                \coordinate (8-1) at (0.309-2,-.951);
                \coordinate (9-1) at (0.809-2,-.588);
                \coordinate (10-1) at (1-2,0);
                
                \coordinate (Center-2) at (0+2, 0);
                
                \coordinate (1-2) at (0.809+2, 0.588);
                \coordinate (2-2) at (0.309+2, 0.951);
                \coordinate (3-2) at (-0.309+2, 0.951);
                \coordinate (4-2) at (-0.809+2,0.588);
                \coordinate (5-2) at (-1+2,0);
                \coordinate (6-2) at (-0.809+2,-.588);
                \coordinate (7-2) at (-0.309+2,-.951);
                \coordinate (8-2) at (0.309+2,-.951);
                \coordinate (9-2) at (0.809+2,-.588);
                \coordinate (10-2) at (1+2,0);

                \draw (Center-1) -- (1-1);
                \draw (Center-1) -- (2-1);
                \draw (Center-1) -- (3-1);
                \draw (Center-1) -- (4-1);
                \draw (Center-1) -- (5-1);
                \draw (Center-1) -- (6-1);
                \draw (Center-1) -- (7-1);
                \draw (Center-1) -- (8-1);
                \draw (Center-1) -- (9-1);
                \draw (Center-1) -- (10-1);

                \draw (Center-2) -- (1-2);
                \draw (Center-2) -- (2-2);
                \draw (Center-2) -- (3-2);
                \draw (Center-2) -- (4-2);
                \draw (Center-2) -- (5-2);
                \draw (Center-2) -- (6-2);
                \draw (Center-2) -- (7-2);
                \draw (Center-2) -- (8-2);
                \draw (Center-2) -- (9-2);
                \draw (Center-2) -- (10-2);
                
                
                \node at (Center-1) [shift={(-0.35,0.5)}] {$\alpha$};
                \node at (Center-2) [shift={(0.35,0.5)}] {$\beta$};

                \draw (8-1) to [out=-72,in=-36] (9-1);
                \draw (10-1) to (5-2);
                
                \draw [dashed, red] (-0.2938-2, -0.4045) to [out=180+36+18, in = 300] (-1.175 - 2, -0.545) to [out=120, in=60+180] (-1.175 - 2, 0.4545) to [out = 60, in=180] (-0.1 - 2, 1.2) to [out=0, in=30](1 - 2, 0.5) to [out=180+30, in=18](0.475 - 2, 0.1545);
                
            
            \end{tikzpicture}
            
            } & $\to$  & \raisebox{-.5\height}{
            
            \begin{tikzpicture}[scale=0.9]
            
                \coordinate (Center-1) at (0-2, 0);
            
            \coordinate (1-1) at (0.809-2, 0.588);
            \coordinate (2-1) at (0.309-2, 0.951);
            \coordinate (3-1) at (-0.309-2, 0.951);
            \coordinate (4-1) at (-0.809-2,0.588);
            \coordinate (5-1) at (-1-2,0);
            \coordinate (6-1) at (-0.809-2,-.588);
            \coordinate (7-1) at (-0.309-2,-.951);
            \coordinate (8-1) at (0.309-2,-.951);
            \coordinate (9-1) at (0.809-2,-.588);
            \coordinate (10-1) at (1-2,0);
            
            \coordinate (Center-2) at (0+2, 0);
            
            \coordinate (1-2) at (0.809+2, 0.588);
            \coordinate (2-2) at (0.309+2, 0.951);
            \coordinate (3-2) at (-0.309+2, 0.951);
            \coordinate (4-2) at (-0.809+2,0.588);
            \coordinate (5-2) at (-1+2,0);
            \coordinate (6-2) at (-0.809+2,-.588);
            \coordinate (7-2) at (-0.309+2,-.951);
            \coordinate (8-2) at (0.309+2,-.951);
            \coordinate (9-2) at (0.809+2,-.588);
            \coordinate (10-2) at (1+2,0);

            \draw (Center-1) -- (1-1);
            \draw (Center-1) -- (2-1);
            \draw (Center-1) -- (3-1);
            \draw (Center-1) -- (4-1);
            \draw (Center-1) -- (5-1);
            \draw (Center-1) -- (8-1);
            \draw (Center-1) -- (9-1);
            \draw (Center-1) -- (10-1);

            \draw (Center-2) -- (1-2);
            \draw (Center-2) -- (2-2);
            \draw (Center-2) -- (3-2);
            \draw (Center-2) -- (4-2);
            \draw (Center-2) -- (5-2);
            \draw (Center-2) -- (6-2);
            \draw (Center-2) -- (7-2);
            \draw (Center-2) -- (8-2);
            \draw (Center-2) -- (9-2);
            \draw (Center-2) -- (10-2);
            
            \node at (Center-1) [shift={(-0.35,0.5)}] {$\alpha$};
            \node at (Center-2) [shift={(0.35,0.5)}] {$\beta$};

            \draw (8-1) to [out=-72, in=-36] (9-1);
            \draw (10-1) to (5-2);

            \draw (6-1) to [out=36, in=0] (3*-0.2938-2, 3*-0.4045) to [out=180, in=60+180] (-1.175 - 2, 0.4545) to [out = 60, in=180] (-0.2 - 2, 1.2) to [out=0, in=30](1 - 2, 0.5) to [out=180+30, in=24](Center-1) to [out=24+180, in=-45] (-2-.1, -0.2) to [out=-45+180, in=180+36] (Center-1);
            
            \draw (7-1) to [out=36*2, in=0] (3*-0.3138-2, 3*-0.4545) to [out=180, in=60+180] (-1.375 - 2, 0.4545) to [out = 60, in=180] (0 - 2, 1.3) to [out=0, in=30] (1.2 - 2, 0.5) to [out=180+30, in=12](Center-1) to [out=12+180, in=180-45] (-2-.1 * 1.7, -0.2 * 1.7) to [out=-45, in=180+72] (Center-1);

            \draw[dashed, red] (-0.475 + 2, 0.1545) to [out=180-18, in=6] (Center-1) to [out=180+6, in=140] (-1.9, -1.2) to [out=-40, in=180+18] (-0.475 + 2, -0.1545);
            
            \end{tikzpicture}}
            \end{tabular}

            }
        \caption{\label{357crossingconnected} Transforming 5-crossings with a crossing segment between face 1 and faces 5 \& 7 into 7-crossings with an ACC.}
\end{figure}

Though we will not show the explicit cases, the same method from Lemma~\ref{first13salvage} can be used to prove two more lemmas: 

\begin{lemma}
Let $D$ be a minimal 5-crossing diagram of knot $K$. Suppose $D$ has crossing $\alpha$ containing at least one monogon. If there exists a crossing segment connecting face~$2$ and any face other than $4$ via $\alpha$, $c_{13}(K)<c_5(K)$.
\end{lemma}

\begin{lemma}
Let $D$ be a minimal 5-crossing diagram of knot $K$. Suppose $D$ has crossing $\alpha$ containing at least one monogon. If there exists a crossing segment connecting face~$3$ and any face other than $5$ via $\alpha$, $c_{13}(K)<c_5(K)$.
\end{lemma}

\begin{remark}
By symmetry, the above lemmas hold for faces 5, 6, and 7. Of the 28 possible crossing segments via $\alpha$, Table~\ref{cstable} categorizes the 6 that do not directly imply $c_{13}(K)<c_5(K)$.
\end{remark}

\begin{table}[! h]
\centering
\caption{Crossing segments via a crossing with a monogon that don't directly imply $c_{13}(K)<c_5(K)$}\label{cstable}
\vspace{4mm}
\begin{tabular}{l|l}
\textbf{First Face} & \textbf{Second Face} \\ \hline

3                   & 5                                       \\ \hline
4                   & 1                                 \\ \hline
4                   & 2                                    \\ \hline
4                   & 6                                   \\ \hline
4                   & 7                                     \\ \hline
4                   & 8                                                \\ 
\end{tabular}
\end{table}

The criteria are less strict for crossings with two almost opposite monogons. There are two ways to notate the faces around a crossing, each with respect to one of the monogons. The only crossing segment that fails using both notations is the crossing segement connecting face 4 and 8. So if any other pair of faces are crossing connected via a crossing with two monogons, $c_{13}(K)<c_5(K)$.\\

It is important to note that aside from face 4, the existence of a crossing segment between faces 1, 2, 6, 7, or 8 and any other face implies $c_{13}(K)<c_5(K)$. Future techniques to prove Conjecture \ref{13ineqconj} could attempt to ``cross out" faces that force a path to travel to a different crossing's face 4. It can already be proved that every 5-crossing diagram must contain two crossing connected faces. So, improving this counting technique to find two crossing connected faces where one of the faces is face 1, 2, 6, 7, or 8 via a crossing with a monogon is satisfactory to prove Conjecture \ref{13ineqconj} for all knots $K$ that contain a minimal 5-crossing projection with at least one monogon. Lastly, extensive work has been put towards finding a counter-example, yet none has been found. All minimal 5-crossing diagrams examined either have a monogon and fall into one of the above categories or contain two OCC faces. It seems likely that Conjecture $\ref{13ineqconj}$ is true. 

\section*{Acknowledgments}
I would like to thank Neel Kolhe, Chengze Li, Sophia Liao, and Eric Tang for helping to initiate this research. Many thanks to Justin Wu for initiating my study of knot theory. I am especially grateful to Colin Adams for advising this research project and first introducing me to $n$-crossing knot diagrams.

\end{document}